\documentclass{article}%{../comput/diacritech_IOS}
\usepackage{etex}
\usepackage{enumerate}
\usepackage{amsmath}
\usepackage{amssymb}

\usepackage{pgfplots}
\pgfplotsset{compat=newest}

\usepackage{amsthm} 
\usepackage{tikz}
\usepackage{multicol}%permet d'afficher du texte en plusieurs colonnes
\usetikzlibrary{positioning,shadows,arrows,automata}
\usepackage[utf8]{inputenc}
\usepackage[T1]{fontenc}
\usepackage{hyperref}
\usepackage{latexsym} 
\usepackage{float}
\usepackage{subcaption}
\usepackage{textcomp}
\usepackage{makeidx}

\makeindex
\usepackage[nottoc,notlof,notlot]{tocbibind} 
\usepackage{complexity} 
\newclass{\MSO}{MSO}
\newtheorem{theorem}{Theorem}[section]
\newtheorem*{theorem*}{Theorem}
\newtheorem*{corollary*}{Corollary}
\newtheorem{lemma}[theorem]{Lemma}  
  
\newtheorem{corollary}[theorem]{Corollary}  
\theoremstyle{definition}
\newtheorem{example}[theorem]{Example}  
\newtheorem{definition}[theorem]{Definition}  
\newtheorem{notation}[theorem]{Notation}  
\newcommand{\quo}[1]{\text{``}#1\text{''}}
\newcommand{\commen}[1]{}

\DeclareMathOperator{\moda}{mod}

\DeclareMathOperator{\lcm}{lcm}

\DeclareMathOperator{\sect}{sec}
\newcommand{\fracInline}[2]{#1/#2}

\newcommand{\mucF}[1]{\mu_{#1}}%dim 1, vrai si R est presb
\newcommand{\minTu}[2]{\min_{#1}\left\{#2\right\}}
\usepackage{framed}
\usepackage{authblk}
\newcommand{\N}{\ensuremath{\mathbb N}{}}
\newcommand{\Z}{\ensuremath{\mathbb Z}{}}

\newcommand{\shift}[4]{pair $(#2,#3)$ can by shifted by $#4$ in $#1$}%1Set, 2position 3taille, 4 vecteur
\newcommand{\shiftable}[4]{pair $(#2,#3)$ is $#4$-shiftable in $#1$}%1Set, 2position 3taille, 4 distance
\newcommand{\notShiftable}[4]{pair $(#2,#3)$ is not $#4$-shiftable in $#1$}%1Set, 2position 3taille, 4 distance
\newcommand{\set}[1]{\left\{#1\right\}}
\newcommand{\cube}[3]{C_{#3}\left(#1,#2\right)}%Le voisinage de 1 à
\newcommand{\ite}[3]{\ensuremath{\left\langle{#1}\mid{#2}\mid{#3}\right\rangle}}
\newcommand{\secti}[3]{\sect(#1; x_{#2}=#3)}% section de 1
\newcommand{\norm}[1]{\ensuremath{\|#1\|}}
\newcommand{\abs}[1]{|{#1}|}% Valeur absolu d'un réel

\newcommand{\mc}{\mathcal}

\newcommand{\tu}{\boldsymbol}

\newcommand{\statement}[2]{\begin{equation}\label{#1}\mbox{\parbox{0.85\textwidth}{#2}}\end{equation}}

\newcommand{\structure}[1]{\mc{#1}}
\newcommand{\vocabulary}[1]{\mc{#1}}

\newcommand{\s}{\structure{S}}
\newcommand{\V}{\vocabulary{V}}
\newcommand{\fo}[1]{\FO\left[#1\right]}%Formule du premier ordre sur
\begin{document}
\title{Uniform definition of sets using relations and complement of
  Presburger Arithmetic.}  \author{Arthur Milchior}% \rhtitle{\tit}

% \title{\tit}
% \date{}
\affil{  %     \begin{multicols}{2}
      %   IRIF, Université Paris 7 - Denis Diderot, France\\
      %   CNRS UMR 8243, Université Paris Diderot - Paris 7, Case 7014\\
      %   75205 Paris Cedex 13\\\columnbreak
        Université Paris-Est Créteil, \\LACL (EA 4219), UPEC,\\
        F-94010 Créteil, France\\
      % \end{multicols}
        {\tt Arthur.Milchior@u-pec.fr}\\
      %  {\tt
      %    \url{http://www.liafa.univ-paris-diderot.fr/web9/equiprech/fichepers_fr.php?id=353}}
     }
     \maketitle{}
\begin{abstract}
  In 1996, Michaux and Villemaire considered integer relations $R$
  which are not definable in Presburger Arithmetic. That is, not
  definable in first-order logic over integers with the addition
  function and the order relation ($\fo{\N,+,<}$-definable
  relations). They proved that, for each such $R$, there exists a
  $\fo{\N,+,<,R}$-formula $\nu_{R}(x)$ which defines a set of integers
  which is not ultimately periodic, i.e. not $\fo{\N,+,<}$-definable.

  It is proven in this paper that the formula $\nu(x)$ can be chosen
  such that it does not depend on the interpretation of $R$. It is
  furthermore proven that $\nu(x)$ can be chosen such that it defines
  an expanding set. That is, an infinite set of integers such that the
  distance between two successive elements is not bounded.
\end{abstract}
% \begin{keywords}
%   First order logic, Presburger arithmetic, definability
% \end{keywords}

\section{Introduction}
This paper deals with first order logic over non-negative integers
with the addition function and the order relation. This logic is
denoted $\fo{\N,+,<}$ and is also called Presburger Arithmetic.  A few
properties of Presburger arithmetic are now recalled.

By \cite{Presburger}, the logic $\fo{\N,+,<}$ admits the elimination
of quantifiers. In particular it implies that $\fo{\N,+,<}$ is a
decidable theory. It is known that sets definable in Presburger
Arithmetic coincide with semilinear sets \cite{ginsburg}, that is,
finite union of linear sets.  A linear set of dimension $d\in\N$ is a
set of the form
$\set{(n^{0}_{0},\dots,n^{0}_{d-1})+\sum_{i=1}^{c}m^{i}(n_{0}^{i},\dots,n_{d-1}^{i})\tu
  \mid m^{1},\dots,m^{c}\in\N}$,
where $c\in\N$ and $n_{j}^{i}\in\N$ for $i\in[c]$ and $j\in[d-1]$.
More precisely, the $\fo{\N,+,<}$-definable sets are exactly the
finite union of disjoint linear sets
\cite{Ito:1969:SSF:1739943.1740133}.  For the particular case of
$d=1$, a set $R$ of integers is $\fo{\N,+,<}$-definable if and only if
it is ultimately periodic sets. That is, if there exists a threshold
$t$ and a period $p>0$ such that, for all $n>t$, $n\in R$ if and only
if $n+p\in R$. 

A characterization of $\fo{\N,+,<}$-definable sets of dimension $d$ is
given in \cite[Theorem 2]{muchnik} in terms of sets of dimension $d-1$
and in terms of local properties. Another characterization of
$\fo{\N,+,<}$-definable sets is given in \cite[Theorem
5.1]{michaux}. This characterization states that for any $d\in\N$, if
a set $R\subseteq\N^{d}$ is not $\fo{\N,+,<}$-definable, then there
exists a set $S$ of integers which is $\fo{\N,+,<,R}$-definable and which is
not $\fo{\N,+,<}$-definable. Furthermore, it can be assumed that $S$
is expanding, that is, that the difference between two successive
elements of $S$ is unbounded.
\paragraph{}
The result \cite[Theorem 5.1]{michaux} is particularly useful to
reduce the complexity of proofs. Instead of assuming that a relation
over integers is not $\fo{\N,+,<}$-definable, it is possible to assume
that a set of integers is $\fo{\N,+,<}$-definable.

In particular, Theorem \cite[Theorem 5.1]{michaux} serves to prove
Cobham's \cite{Cobham} and Semenov's \cite{semenov-theorem} Theorem
(Cobham's Theorem being the case $d=1$ of Semenov's
Theorem). Semenov's Theorem is: \quo{Let $k$ and $l$ be
  multiplicatively independant (i.e. have no common power apart from
  1). If $R\subseteq\N^{d}$ is definable in $\fo{\N,+,<,V_{k}}$ and in
  $\fo{\N,+,<,V_{l}}$ then $R$ is $\fo{\N,+,<}$-definable.}  Here
$V_{m}$ is the function which maps every nonzero natural number to the
greatest power of $m$ dividing it.

\paragraph{}
The \quo{central idea} of \cite[Theorem 5.1]{michaux}, as stated in
\cite[page 272]{michaux}, is the following. Let $R\subseteq\N^{d}$ be
a relation which is not $\fo{\N,+,<}$-definable. Furthermore, assume
that all sets of integers which are $\fo{\N,+,<,R}$-definable are
ultimately periodic. It can be shown that the negation of Muchnik's
characterization \cite[Theorem 2]{muchnik} of $\fo{\N,+,<}$-definable
sets does not hold (i.e. that Muchnik characterization holds). Since
Muchnik's characterization holds, it implies that $R$ is
$\fo{\N,+,<}$-definable, which contradicts the hypothesis.

A careful analysis of the proof of \cite[Theorem 5.1]{michaux} shows
that the use of proof by contradiction can be avoided. Removing the
proof by contradiction would lead to a method which, given a relation
$R\subseteq\N^{d}$, allows to construct a formula
$\nu_{R}(x)\in\fo{\N,+,<,R}$, defining a set of integers which is not
$\fo{\N,+,<}$-definable and which is expanding.

In this paper, we prove that this formula $\nu_{R}(x)$ can be chosen
independently of $R$. That is, there exists a $\fo{\N,+,<,R}$-formula
$\nu_{d}(x)$, such that, if $R$ is not $\fo{\N,+,<}$-definable, then
$\nu_{d}(x)$ defines an expanding set of integers, hence a set of
integers which is not $\fo{\N,+,<}$-definable.

\paragraph{}
Standard definitions are recalled in \autoref{sec:def}. Results
related to $\fo{\N,+,<}$ are recalled in \autoref{sec:some-res}.  The
main theorem is stated and proved in \autoref{mva}.
% \paragraph{}

% While there is no known application of this result, it can be imagined
% that this formula could be used as a piece of another result, where
% the set $R$ itself is defined by a logical formula with parameters. 
\section{Definitions}\label{sec:def}

In this section, definitions are recalled.  To avoid ambiguity, the
\quo{=} symbol is used for mathematical equality. The symbol \quo{:=}
is used when terms are defined. And the equality relation in formulas
is denoted by \quo{$\doteq$}.  Let $\Z$ denote the set of integers,
let $\N$ denote the set of non-negative integers and let $\N^{>0}$
denote the set of positive integers. For $a\in \Z$, let $\abs{a}$
denote the absolute value of $a$, that is, $a$ if $a\in\N$ and $-a$
otherwise. For $S$ a finite set of positive integers, let
$\lcm(S)$\index{lcm@$\lcm(S)$} denote the least common multiple of the
element of $S$, that is, the least integer $n$ such that, for each
$i\in S$, $i$ divides $n$.

For $a\in \N$, let $[a]$ denote $\set{0,\cdots,a}$. For $d\in\N^{>0}$,
let $S^{d}$ denote the set of $d$-tuples of elements of $S$. Let bold
letters denote $d$-tuples of variables, such as $\tu{x} \in\N^{d}$,
which is an abbreviation for $(x_{0},\cdots,x_{d-1})$. For
$i\in[d-1]$, the variable $x_{i}$ is called the $i$-th component of
$\tu{x}$.  Let $\max(\tu{x})$ denote $\max\set{x_{i}\mid i\in[d-1]}$,
let $\min(\tu{x})$ denote $\min\set{x_{i}\mid i\in[d-1]}$ and let
$\norm{\tu x}$ denote $\sum_{i=0}^{d-1}x_{i}$, it is said to be the
norm of $\tu x$.

Functions and relations are applied component-wise on $d$-tuples. In
particular $\tu x<\tu y$ means that $x_{i}<y_{i}$, for all
$i\in[d-1]$. Let $\abs{\tu x}$ (respectively, $\tu x+\tu y$) denote
the $d$-tuple $\left(\abs{x_{0}},\dots,\abs{x_{d-1}}\right)$
(respectively, $\left(x_{0}+y_{0},\dots,x_{d-1}+y_{d-1}\right)$). Let
$\tu f(n)$ denote $\left(f_{0}(n),\dots,f_{d-1}(n)\right)$ for
$n\in\N$.

\begin{definition}[Ultimately ($m$-)periodic]\label{def:ult-per}
  A set $R\subseteq\N$, is \emph{ultimately
    $m$-periodic}\index{Ultimately $m$-periodic} if there exists an
  integer $t\in\N$ such that for all $n\ge t$, $n\in R$ if and only if
  $n+m\in R$. A set is said to be ultimately periodic if it is
  ultimately $m$-periodic for some $m\in\N^{>0}$. The least such integer $t$
  is called the threshold of $R$. The least such integer $m$ is called
  the minimal period of $R$.
\end{definition}
\begin{definition}[Expanding set]\label{def:ult-per}
  A set $R\subseteq\N$, is \emph{expanding}\index{Expanding set of
    integers} if it is infinite and if the distance between two
  successive integers belonging to $R$ is not bounded.
\end{definition}

\subsection{First-order logic}
In this section, the definitions concerning the logical formalism of
this paper are introduced.

\begin{definition}[Vocabulary]\index{Vocabulary}
  A \emph{vocabulary} is a set of the form
  \begin{equation*}
    \V=\set{(R_{i}/d_{i})_{i<n},(f_{i}/d'_{i })_{i<p},(c_{i})_{i<q}},
  \end{equation*}
  where $n$, $p$ and $q$ are either integers or $\omega$ (the
  cardinality of the set of integers).
  
  For $i<n$, the $R_{i}$ is a \emph{relation} symbol and its arity is
  $d_{i}$. For $i<p$, the $f_{i}$ is a \emph{function} symbol and its
  arity is $d'_{i}$.  For $i<q$, the $c_{i}$ is a \emph{constant}
  symbol.
\end{definition}
In this paper the value of $p$ is $1$, apart in Lemma
\ref{lem:croissant}, and the only function is the addition.  The value
of $n$ is $1$ or $2$ relation. The relations considered in this paper
are the order relation $<$ and a relation $R$ of dimension
$d\in\N^{>0}$.

\begin{definition}[Structure]\index{Structure}\label{def:structure}
  Let $\V$ be a vocabulary. A $\V$-\emph{structure} $\s$ over the
  universe $\N$ is a tuple
  \begin{equation*}
    (\N,(R_{i}^{\s})_{i<n},(f_{i}^{\s})_{i<p},(c^{\s}_{i})_{i<q})
  \end{equation*}
  where $R_{i}^{\s}\subseteq \N^{d_{i}}$ for $i<n$, where
  $f_{i}^{\s}:\N^{d'_{i}}\to \N$ for $i<p$, and where $c^{\s}_{i}\in
  \N$ for $i<q$.

  For every constant symbol $x$, and $c\in\N$, let $\s[x/c]$ denote
  the structure such that $x^{\s[x/c]}=c$, and
  $\varsigma^{\s[x/c]}=\varsigma^{\s}$ for all other symbols
  $\varsigma\in\V\setminus\set{x}$.
\end{definition}
In this paper, we consider the standard interpretation of $+$ and $<$
over $\N$.  The first-order logic used in this paper is now defined.
\begin{definition}[$\fo{\N,\V}$]
  \label{formula} 
  The set of $\V$\emph{-terms} is defined by the grammar:
  \begin{eqnarray*}
    t(\V)::=c_{i}\mid f_{i}(t_{0},\dots,t_{d'_{i}-1})
  \end{eqnarray*}
  where $c_{i}$ is a constant of $\V$, $f_{i}$ is a function of $\V$
  and the $t_{j}$'s are $\V$-terms.

  The first-order logic over the vocabulary $\V$, denoted by
  $\fo{\N,\V}$, is defined by the grammar:
  \begin{eqnarray*}
    \fo{\N,\V}::=
    \exists x.\psi
    \mid\forall x.\psi
    \mid\neg\phi_{0}
    \mid\phi_{0}\land\phi_{1}
    \mid\phi_{0}\lor\phi_{1}
    \mid{}R_{i}\left(t_{0},\dots,t_{d_{i}-1}\right)
    \mid{}t_{0}\doteq{t_{1}}
  \end{eqnarray*}
  \index{$\land$}\index{$\neg$}
  \index{$\lor$}
  \index{R@$R_{i}\left(t_{0},\dots t_{d_{i}-1}\right)$} where
  the $t_{i}$'s are $\V$-term, $R_{i}$ is a symbol belonging to $\V$,
  the $\phi_{i}$'s are $\fo{\N,\V}$-formulas and $\psi$ is a
  $\fo{\N,\V,x}$-formula.
\end{definition}
The atomic formula $<(x,y)$ is denoted $x<y$.  Let
$\phi_{0}\implies\phi_{1}$\index{$\implies{}$} be an abbreviation for
$\left(\neg\phi_{0}\right)\lor\phi_{1}$ and let $\phi_{0}\iff\phi_{1}$
be an abbreviation for
$\left(\phi_{0}\implies\phi_{1}\right)\land\left(\phi_{1}\implies\phi_{0}\right)$.
The dimension and the curly brackets are omitted in logics'
notations. For instance $\fo{\N,\set{+/2,</2}}$ is abbreviated in
$\fo{\N,+,<}$.  Let $\phi$ be a
$\fo{\N,\V,x_{0},\dots,x_{d-1}}$-formula. Then
$\phi(x_{0},\dots,x_{d-1})$ is said to be an $\fo{\N,\V}$-formula with
dimension $d$.  The $x_{i}$'s for $i\in[d-1]$, are called the free
variables and do not belong to $\V$.  Given some $\V$-structure $\s$,
the semantic of a $\fo{\N,\V}$-formula is defined recursively as usual.
\begin{definition}[Definability]\index{Definability}\label{def:definability}
  Let $\V$ be a vocabulary and $\s$ be a $\V$-structure.  Let $d\in\N$
  and $\phi(x_{0},\dots,x_{d-1})$ be a $\fo{\N,\V}$-formula of dimension
  $d$.  The formula $\phi$ is said to \emph{define} the $d$-ary set
  \index{$\phi(\tu{x})^{\s}$ - set defined by a formula}
  $\set{\tu{n}\in \N^{d}\mid\s[\tu x/\tu n]\models \phi(\tu x)}$ in
  $\s$.  A set $R\subseteq \N^{d}$ is said to be $\fo{\N,\V}$-definable
  in $\s$ if there exists $\phi(x_{0},\dots,x_{d-1})\in \fo{\N,\V}$ such
  that $R=\phi(x_{0},\dots,x_{d-1})^{\s}$.
\end{definition}

\subsection{Some notations}
Some notations are introduced in this section in order to simplify
creation of formulas.  A notation is now introduced which allows to
simplify the logical definitions of functions.
\begin{notation}\label{fun-for}
  Let $d,d'\in\N$, and let $\V$ be a vocabulary.  Let:
  \begin{equation*}
    \phi(x_{0},\dots,x_{d-1};y_{0},\dots,y_{d'-1})
  \end{equation*}
  \index{$\phi(\tu{x}; \tu y)$ - function notation}denote that, for
  every $d$-tuple $\tu n\in\N^{d}$, there exists exactly one
  $d'$-tuple $\tu{n'}\in\N^{d'}$ such that
  $\s[\tu{x}/\tu{n}][\tu{y}/\tu{n'}]\models\phi(x_{0},\dots,x_{d-1},y_{0},\dots,y_{d'-1})$.
  Then the $d'$-tuple $\tu{n'}$ is denoted by $\phi(\tu{n})$.  More
  precisely, for $\psi(\tu y)$ a formula with $d'$ free variable,
  $\psi(\phi(\tu n))$ is an abbreviation for
  $\exists \tu{n'}.\phi(\tu n,\tu{n'})\land\psi(\tu{n'})$.
\end{notation}

The following notation states that some variables are interpreted by
the minimal value such that a formula holds.

\begin{notation}\label{min-exists}
  Let $F$ be a finite ordered set. Let $(\phi_{i})_{i\in{F}}$ be a set
  of $\fo{\N,\V}$-formulas.  Let $i\in F$.  A $\fo{\N,\V}$-formula
  $\minTu{i}{\phi_{i}}$ is introduced
  \index{min@$\minTu{i\in{}F}{\phi_{i}}$ - $i$ is minimal such that
    $\phi_{i}$} which states that $\phi_{i}$ holds and $i$ is minimal
  with this property. Let:
  \begin{equation*}
    \minTu{i}{\phi_{i}}:=\phi_{i}\land\bigwedge_{j\mid\ j<i}\neg \phi_{j}.
  \end{equation*}

  Similarly, let $\tu{x}=(x_{0},\dots, x_{d-1})$ be a tuple of
  variables and let $\phi(\tu{x})$ be a $\fo{\N,\V}$-formula. A
  $\fo{\N,\V,<}$-formula $\displaystyle \minTu{\tu{x}}{\phi_{i}(\tu{x})}$
  is introduced, \index{min@$\minTu{\tu{x}}{\phi(\tu{x})}$ - $\tu x$
    is minimal such that $\phi(\tu{x})$} which states that
  $\phi(\tu{x})$ holds and $\tu{x}$ is lexicographically minimal with
  this property. Let:
  \begin{equation*}
    \minTu{\tu{x}}{\phi(\tu{x})}:=\phi(\tu{x})\land
    \forall\tu{y}.\left\{\left[\bigvee_{j=0}^{d-1}\left(y_{j}<x_{j}\land\bigwedge_{k=0}^{j-1}y_{k}\doteq{}x_{k}\right)\right]\implies{\neg\phi(\tu{y})}\right\}.
  \end{equation*}

  Finally, for $\phi_{i}$ a family of $\fo{\N,\V}$-formulas, let
  $\minTu{i,\tu{x}}{\phi_{i}(\tu{x})}$ \index{min@
    $\minTu{i\in{}F,\tu{x}}{\phi_{i}(\tu{x})}$ - $i,\tu x$ is minimal
    such that $\phi_{i}(\tu{x})$} be a $\fo{\N,\V,<}$-formula which
  states that $\phi_{i}(\tu{x})$ holds and $(i,\tu{x})$ is
  lexicographically minimal with this property. Let:
  \begin{equation*}
    \minTu{i,\tu{x}}{\phi_{i}(\tu{x})}:=
    \minTu{i}{\exists \tu{y}.\phi_{i}(\tu{y})}\land
    \minTu{\tu{x}}{\phi_{i}(\tu{x})}.
  \end{equation*}
\end{notation}
An example of formula using this notation is now given.
\begin{example}
  Let $R$ be a unary relation symbol. Let
  $\phi(x):=\minTu{x}{R(x)\land \neg R(x+1)}$. This formula state that
  $x$ is the last element of the least sequence of successive elements
  of $R$.
\end{example}

Notations for implications and equivalences are standard. A notation
of the form \quo{if then else} is also needed. It is now introduced.
\begin{notation}\label{ite}
  Let $F$ be a finite set. Let $\psi$ be a $\fo{\N,\V}$-formula. For
  $i\in F$, let $\phi_{i}(\tu{x})$ and $\chi_{i}(\tu{x})$ be
  $\fo{\N,\V}$-formulas. Let \index{$\ite{\bigvee_{i\in F}\exists\tu
      x.\phi_{i}(\tu{x})}{\chi_{i}(\tu{x})}\psi$ - If $\phi$ then
    $\chi$ else $\psi$}
  \begin{eqnarray*}
    \ite{\bigvee_{i\in{F}}\exists\tu{x}.\phi_{i}(\tu{x})}{\chi_{i}(\tu{x})}\psi
  \end{eqnarray*}
  be a $\fo{\N,\V}$-formula which states that if there exists $i\in F$
  and $\tu n\in\N^{d}$ such that $\phi_{i}(\tu{n})$ holds, then
  $\chi_{i}(\tu{n})$, otherwise $\psi$.  Formally, the formula is:
  \begin{eqnarray*}
    \left\{
      \bigvee_{i\in{F}}\exists\tu{x}.\phi_{i}(\tu{x})
      \land
      \chi_{i}(\tu{x})
    \right\}
    \lor
    \left\{
      \bigwedge_{i\in F}\forall\tu{x}.\neg\phi_{i}(\tu{x})
      \land
      \psi
    \right\}.
  \end{eqnarray*}
\end{notation}
An example of formula using this notation is now given.
\begin{example}
The formula
\begin{equation*}
  \ite{\bigvee_{i=3}^{5}x\doteq i}{\exists z. z\times
    i\doteq y}{\exists z. 2z+1\doteq y}
\end{equation*}
states that if $x$ is 3, 4 or 5, then $y$ is a multiple of $x$,
otherwise $y$ is odd.
\end{example}
In this paper, the two preceding notations are used together, stating
that, if there are some $i\in F$ and $\tu x\in \N^{d}$ such that
$\phi_{i}(\tu{x})$ holds, the minimal pair is considered in
$\chi_{i}(\tu{x})$, otherwise the formula $\psi$ is considered.

\section{Some results about $\fo{\N,+,<}$-definable sets}\label{sec:some-res}
In this section, theorems concerning $\fo{\N,+,<}$-definable sets of
dimension $d>0$ are recalled. The theorems concerning any positive
dimension are given in Section \ref{sec:theo-d>0}, and the theorem
concerning the dimension $1$ are given in Section \ref{sec:theo-d=1}.

\subsection{Positive dimension}
\label{sec:theo-d>0}
The theorem given in this section is a variant of the characterization
of $\fo{\N,+,<}$-definable relations given in \cite[Theorem
1]{muchnik}. This presentation of the theorem is inspired of
\cite[Theorem 5.5]{michaux}.  This characterization of
$\fo{\N,+,<}$-definable sets consists in two properties, a local
property and a recursive property.  The recursive property of
\cite[Theorem 5.5]{michaux} uses the notion of section, which is now
defined.
\begin{definition}[Section]
  \index{Section of a susbet of $\N^{d}$} Let the dimension $d$ be at
  least 2, $R\subseteq\N^{d}$, $i\in[d-1]$ and $c\in\N$. Then the
  \emph{section of $R$ in $x_{i}=c$}, denoted by $\secti{R}{i}{c}$, is
  the set of $(d-1)$-tuples obtained from $R$ by fixing the $i$th
  component to $c$, that is: \index{Section@$\secti{R}{i}{c}$ -
    Section $x_{i}=c$ of $R$}
  \begin{equation*}
    \secti{R}{i}{c}:=\set{(x_{0},\dots,x_{d-2})\in \N^{d-1}\mid{}(x_{0}, \dots, x_{i-1},c,x_{i},\dots, x_{d-2})\in R}.
  \end{equation*} 
\end{definition}
The sections of the addition relation are now given as examples.
\begin{example}\label{ex:subspace}
  Let $R=\set{(x_{0},x_{1},x_{2})\mid x_{0}+x_{1}=x_{2}}$. Its
  sections are now studied. Let $c\in\N$. One has:
  \begin{eqnarray*}
              \secti   {R}{0}{c}     =\secti   {R}{1}{c}     =\{(n,    n+c )\mid{}n\in\N\},\\    \secti{R}{2}{c}   =\{(n,c-n )\mid{}n\le c\},\\
  \end{eqnarray*}
\end{example}

The local property of \cite[Theorem 5.5]{michaux} uses the notion of
cube which is now introduced.
\begin{notation}
  Let $d\in\N$, $R\subseteq\N^{d}$, $\tu{x}\in\N^{d}$ and
  $k\in\N$. The $R$\emph{-cube} at $\tu x$ of size $k$, denoted by
  $\cube{\tu{x}}{k}R$\index{crxk@$\cube{\tu{x}}{k}R$ - The cube of $R$
  of radius $k$ around $\tu x$}, is defined as:
  \begin{equation*}
    \cube{\tu{x}}{k}R:=\set{\tu{y}\in[k]^{d}\mid{}\tu x+\tu y\in R}
  \end{equation*}
\end{notation}
The following lemma considers equality of cubes.
\begin{lemma}\label{lem:eq-cube}
  Let $d\in\N$, and $R$ be a $d$-ary relation symbol. Let $\s$ be a
  $\set{\N,+,<,R}$-structure. There exists a $\fo{\N,+,<,R}$-formula
  $\beta_{d}(\tu{x},\tu y,k)$\index{bd@$\beta_{d}(\tu{x},\tu y,k)$ -
    The cube of radius $k$ at positions $\tu x$ and $\tu y$ are
    equals.} which states that the cubes $\cube{\tu{x}}{k}{R^{\s}}$
  and $\cube{\tu{y}}{k}{R^{\s}}$ are equal.
\end{lemma}
\begin{proof}
  The formula is:
  \begin{equation}\label{eq:beta}
    \beta_{d}(\tu{x},\tu y,k):=  \forall \tu z.
    \left[
      \max(\tu z)\le k
    \right]
    \implies
    {
      \left[R(\tu{x}+\tu z)\iff R(\tu
        y+\tu z )\right]
    },
  \end{equation}
  where $\max(\tu z)<k$ denotes $\bigwedge_{i=0}^{d-1}z_{i}<k$.
\end{proof}
The local property of \cite[Theorem 5.5]{michaux} also uses the
notion of \emph{shifting} a cube. This notion is now introduced.
\begin{definition}
  Let $d\in \N$, $R\subseteq\N^{d}$, $\tu x\in\N^{d}$ and
  $\tu r\in\Z^{d}$. Then it is said that the pair $(\tu x,k)$ can be
  \emph{shifted by $\tu r$} in $R$ if
  $\cube{\tu{x}}{k}R=\cube{\tu{x}+\tu r}{k}R$.
\end{definition}
It is now explained how to state in first order logic that the pair
$(\tu x,k)$ can be {shifted by $\tu r$} in $R$.
\begin{lemma}\label{lem:eq-cube}
  Let $d\in\N$, and $R$ be a $d$-ary relation symbol. Let $\s$ be a
  $\set{\N,+,<,R}$-structure. There exists a $\fo{\N,+,<,R}$-formula
  $\sigma_{d}(\tu r,k,\tu{x})$\index{sd@$\sigma_{d}(\tu r,k,\tu{x})$ -
    The \shift{R^s}{\tu x}{k}{\tu r}} which states that the
  \shift{R^s}{\tu x}{k}{\tu r}.
\end{lemma}
\begin{proof}
  The formula is:
  \begin{equation}\label{eq:sigma}
    \sigma_{d}(\tu r,k,\tu{x}):=\beta_{d}(\tu{x},\tu x+\tu r,k).
  \end{equation}
  where $\beta_{d}$ is the formula of Lemma \ref{lem:eq-cube}.
  Note that $\tu r\in\Z^{d}$. Formally, the variables takes values in
  $\N$ in our formalism.  It is trivial to simulate such variables
  taking values in $\Z$ by using twice as many variables
  $\tu r^{0},\tu r^{1}\in\N^{d}$ and considering $\tu r$ as
  $\tu r^{0}-\tu r^{1}$.
\end{proof}
The notion of pairs which admits a shift whose norm is bounded by some
constant $s$ is now introduced.
\begin{definition}
  Let $d\in\N^{>0}$, $R\subseteq\N^{d}$, $\tu x\in\N^{d}$, and $s\in\N$. If
  there exists $\tu r\in\Z^{d}\setminus\left(0,\dots,0\right)$ such
  that $\max(|r|)\le s$ and such that the \shift{R}{\tu x}{k}{\tu r},
  then the pair $(\tu x,k)$ is said to be $s$-shiftable in $R^{\s}$.
\end{definition}
It is now explained how to state in first order logic that the pair
$(\tu x,k)$ is $s$-shiftable in $R$.
\begin{lemma}\label{lem:s-shiftable}
  Let $d\in\N$, and $R$ be a $d$-ary relation symbol. Let $\s$ be a
  $\set{\N,+,<,R}$-structure. There exists a $\fo{\N,+,<,R}$-formula
  $\varsigma_{d}(s,k,\tu{x})$\index{sd@$\varsigma_{d}(s,k,\tu{x})$ -
    The \shift{R^\s}{\tu x}{k}{s}} which states that the
  \shiftable{R^\s}{\tu x}{k}{s}.
\end{lemma}
\begin{proof}
  The formula is:
  \begin{equation}\label{eq:varsigma}
    \varsigma_{d}(s,k,\tu{x}):=    \exists \tu r\in\Z^{d}.
    \max(|\tu r|)\le
    s\land\bigvee_{i=0}^{d-1}r_{i}\not\doteq{}0 \land \sigma_d(\tu r,k,\tu{x}).
  \end{equation}
  where $\sigma_{d}$ is the formula of Lemma \ref{lem:eq-cube}.
\end{proof}
\paragraph{}

A variant of Muchnik's theorem is now recalled.
\begin{theorem}[{\cite[Theorem 5.5]{michaux}}]
  \label{theo-muchnik-original}
  Let $d\in\N^{>0}$ and $R\subseteq\N^{d}$. The following properties are
  equivalent;
  \begin{enumerate}
  \item The set $R$ is $\fo{\N,+,<}$-definable.
  \item 
    \begin{enumerate}[(a)]
    \item\label{muc-recur} If the dimension $d$ is at least 2, then
      all sections of $R$ are $\fo{\N,+,<}$-definable and
    \item\label{muc-local} there exists $s\in\N$ such that, for every
      $k\in\N$, there exists $t\in\N$ such that, for all
      $\tu{c}\in\N^{d}$ with $t\le\min(\tu{c})$, the \shift{R}{\tu
        x}{k}{s}.
    \end{enumerate}
  \end{enumerate}
\end{theorem}
Property \eqref{muc-local} is now commented. Intuitively, the integer
$s$ represents the bound on the norm of the shift, the integer $k$
represents the size of the cube, the integer $t$ represents a
threshold and the $d$-tuple $\tu{c}$ represents the lowest corner of
the cubes considered. Property \eqref{muc-local} states that there
exists a distance $s$ such that for all sizes $k$ of cubes, there
exists cubes of size $k$ whose components are arbitrary great and this
cube is $s$-shiftable

Let us say a word about the difference between \cite[Theorem
2]{muchnik}, \cite[Theorem 5.5]{michaux} and Theorem
\ref{theo-muchnik-original}. The version of \cite[Theorem 2]{muchnik}
considers a notion of periodicity, while \cite[Theorem 5.5]{michaux}
and Theorem \ref{theo-muchnik-original} consider a notion of
shift. Each of those notion can be restated using the other
notion. The version of \cite[Theorem 2]{muchnik} consider relations
over $\Z$ while \cite[Theorem 5.5]{michaux} and Theorem
\ref{theo-muchnik-original} only considers relations over $\N$. The
condition about $\tu c$ in Property \eqref{muc-local} of Theorem
\ref{theo-muchnik-original} is \quo{$t\le\min(\tu{c})$}, while it is
\quo{$t\le\norm{\tu{c}}$} in \cite[Theorem 2]{muchnik} and it is
\quo{$t\le\max(\tu{c})$} in \cite[Theorem 5.5]{michaux}.  The proof of
\cite[Theorem 5.5]{michaux} still holds when $t\le\max(\tu{c})$ is
replaced by $t\le\min(\tu{c})$ or by $t\le\norm{\tu c}$.

One of the main interest of Theorem \ref{theo-muchnik-original} is given
in the following Theorem.
\begin{theorem}[{\cite[Theorem 2]{muchnik}}] \label{much-orig-formula}
  There exists a $\fo{\N,+,<,R}$-formula
  $\mucF{d}$\index{mud@$\mucF{d}$ - The formula which states that $R$
    is $\fo{\N,+,<}$-definable} such that, for every
  $\set{\N,+,<,R}$-structure $\s$, $\s\models\mucF{d}$ if and only if
  $R^{\s}$ is $\fo{\N,+,<}$-definable.
\end{theorem}

A corollary of Theorem \ref{theo-muchnik-original} is now given.
\begin{corollary}\label{theo-muchnik}
  Let $d\in\N^{>0}$ and $R\subseteq\N^{d}$. If $R$ is not $\fo{\N,+,<}$-definable, then one of
  the two following statements hold:
  \begin{enumerate}[(a)]
  \item\label{muc-ori-recur} the dimension is at least 2 and a section
    of $R$ is not $\fo{\N,+,<}$-definable or
  \item\label{muc-ori-local} for every $s\in \N$, there exists
    $k(R,s)\in\N$ such that for every $t\in\N$ there exists
    $\tu{c}(R,s,t)\in\N^{d}$ with $t\le\min(\tu{c}(R,s,t))$ such that
    the \notShiftable{R^\s}{\tu{c}(R,s,t)}{k(R,s)}{s}.
  \end{enumerate}
\end{corollary}

Let $\s$ be a $\set{\N,+,<,R}$-structure such that $R^{\s}$ is not
$\fo{\N,+,<}$-definable. Let $s,t\in\N$. There may exist many values
$k(R^{\s},s)$ and $\tu{c}(R^{\s},s,t)$ for wich Property
\eqref{muc-ori-local} holds. In this paper, it is always assumed that
$k(R^{\s},s)$\index{kr@$k(R^{\s},s)$ - the least radius $k$ according
  to Muchnik's theorem} and
$\tu{c}(R^{\s},s,t)$\index{cr@$\tu{c}(R^{\s},s,t)$ - the least cube
  $c$ according to Muchnik's theorem} represent the lexicographically
minimal such values.

Two examples of applications of this corollary
are now given.  \newcommand{\forExDeux}{x_{1}\equiv1\mod 2, x_{0}\le
  x_{1}^{2}}
\begin{example}\label{ex:muc}
  Let $R_{0}=\set{(x_{0}^{2},x_{1})\mid x_{0},x_{1}\in\N}$.  In this
  case, Property \eqref{muc-ori-recur} of Corollary
  \ref{theo-muchnik} clearly holds, for the section $x_{1}=0$.
\end{example}
\begin{example}\label{ex:muc2}
  Let $R_{1}=\set{(x_{0},x_{1})\in\N^{2}\mid \forExDeux}$. This set is
  pictured in Figure \ref{fig:ex:muc}.
    \begin{figure}[h]
      \begin{tikzpicture}[darkstyle/.style={circle,draw,minimum
          size=1},scale=.42]
        % les bases
        \draw [ -latex] (0,0) -- (0,6);%x_0 axis
        \draw [ -latex] (0,0) -- (26,0);%x_1 axis
      
        \node at (-1.1,5) {$x_1$}; \node at (25,-1.1) {$x_0$};

        % Les points
        \foreach \xo in {1,3,...,5}{
          \pgfmathsetmacro{\xiMax}{\xo *\xo }
          \foreach \xi in {0,1,...,\xiMax} {
            \draw [fill=black]  (\xi,\xo) circle  (2.5pt);
          }
        }  
      
        % Les cube
        \foreach \up in {1,3,...,5}{
          \pgfmathsetmacro{\left}{\up *\up }
          \pgfmathsetmacro{\right}{\left+1 }
          \pgfmathsetmacro{\down}{\up-1 }
          \draw (\left,\up) -- (\right,\up) -- (\right,\down) -- (\left,\down) -- (\left,\up);
        }        
        
        % indices x_0 vertical
        \foreach \xo in {0,1,...,6}{
          \node at (-.5,\xo) {$\xo$};
        }
        % indices x_1 horizontal
        \foreach \xi in {0,1,...,26}{
          \node at (\xi,-.4) {$\xi$};
        }
      \end{tikzpicture}
      \caption{$R_{1}=\set{(x_{0},x_{1})\in\N^{2}\mid \forExDeux}$ from
        Example \ref{ex:muc}}
      \label{fig:ex:muc}
    \end{figure} 
    In this case, Property \eqref{muc-ori-recur} does not hold and
    Property \eqref{muc-ori-local} holds. For every $s\in\N^{>0}$, it
    suffices to consider cubes of size 1, that is $k(R,s)=1$. Indeed,
    there is an infinite number of $\tu x\in\N^{d}$ such that
    $\cube{\tu{x}}{1}R$ equal to $\set{(0,1)}$ and such that the
    \notShiftable{R^\s}{\tu x}{1}{s}. For small values
    of $s$, some of those cubes are shown in Figure \ref{fig:ex:muc}.

    More precisely, for every $t\in\N$ and for every $s\in\N^{>0}$,
    $\tu{c}(R^{\s},s,t)$ equals $((n+1)^{2},n)$ where $n$ is the least
    integer greater or equal to $\max\left(t,\fracInline{s}{4}\right)$.
    
    % The first possible values for the pair $((c+1)^{2},c)$ are
    % represented by the bottom-left of squares of Figure
    % \ref{fig:ex:muc}. It should be noted that
    % $\cube{((c+1)^{2},c)}1R=\set{(0,1)}$. And the pair
    % $(n,n')\in\N^{2}$ which is the closest from $((c+1)^{2},c)$ and
    % such that $\cube{(n,n')}1R=\set{(0,1)}$, is $((c-1)^{2},c-2)$. And
    % $\max(((c+1)^{2},c)-((c-1)^{2},c-2))=\max(4c,2)>
    % 4(\frac{s}{4})=s$. Hence, $((c+1)^{2},c)$ is not $s$-shiftable.
\end{example}

The following lemmas allow to define the functions $k$ and $\tu{c}$ as
first-order formulas which do not depend of the interpretation of $R$.
\begin{lemma}\label{lem:k}
  Let $d>0$, and $R$ be a $d$-ary relation symbol.  Let $\s$ be a
  $\set{\N,+,<,R}$-structure. There exists a $\fo{\N,+,<,R}$-formula
  $\kappa_{d}(s;K)$ which states that $K= k(R^{\s},s)$ if
  $k(R^{\s},s)$ is correctly defined.
\index{kd@$\kappa_{d}(s;K)$ - the formula which
  defines $k(R,s)$}
\end{lemma}
\begin{proof}
  The formula $\kappa_{d}(s;K)$ states that $K=k(R^{\s},s)$.  The
  integer $k(R^{\s},s)$ is the minimal integer such that, for all
  $t\in\N$, there is a $d$-tuple $\tu{c}\in\N^{d}$ with
  $t\le\min(\tu{c})$ such that the \notShiftable{R^\s}{\tu
    c}{k(R^\s,s)}{s}.  Let:
  \begin{equation}
    \kappa_{d}(s;K):=
    \minTu{K}{\forall t.\exists \tu{c}.t\le\min(\tu{c})\land
      \neg\varsigma_{d}(s,K,\tu{c})},
  \end{equation}
  where $\varsigma_{d}(s,K,\tu c)$ is the formula of
  Lemma \ref{lem:s-shiftable} which states that the \shiftable{R^\s}{\tu
    c}{K}{s}. Recall that the notation $\minTu{K}{\phi}$ is introduced
  in Notation \ref{min-exists}.
\end{proof}  
\begin{lemma}\label{lem:c}
  \index{gd@$\gamma_{d}(s,t;\tu{c})$ - the formula which defines
    $c(R,s,t)$} Let $d>0$, and $R$ be a $d$-ary relation symbol.  Let
  $\s$ be a $\set{\N,+,<,R}$-structure. There exists a
  $\fo{\N,+,<,R}$-formula $\gamma_{d}(s,t;\tu{C})$ which states that
  $\tu{C}=\tu{c}(R^{\s},s,t)$ if $\tu{c}(R^{\s},s,t)$ is defined.
\end{lemma}
\begin{proof}
  The formula $\gamma_{d}(s,t;\tu{C})$ states that $\tu{C}$ is
  lexicographically minimal such that:
  \begin{itemize}
  \item $t\le\min(\tu{C})$ and
  \item the \notShiftable{R^\s}{\tu C}{k(s)}{s}.
  \end{itemize}
  Let
  \begin{equation}\label{eq:gamma}
    \gamma_{d}(s,t;\tu{C}):=\minTu{\tu{C}}{\bigwedge_{i=0}^{d-1}t\le C_{i}\land\neg\varsigma_{d}(s,\kappa_{d}(s),\tu{C})},
  \end{equation} where $\varsigma_{d}(s,\kappa_{d}(s),\tu{C})$ is the
  formula of Lemma \ref{lem:s-shiftable}.
\end{proof}

\subsection{Dimension $d=1$}
\label{sec:theo-d=1}
Two theorems dealing with set of integers and the logic $\fo{\N,+,<}$
are recalled in this section.
\begin{theorem}[\cite{Presburger}]\label{relun}
  A set $R\subseteq\N$ is $\fo{\N,+,<}$-definable if and only if it is ultimately
  periodic.
\end{theorem}
\begin{theorem}[{\cite[Theorem 3.7]{michaux}}]\label{theo:exp}
  Let $R$ be a unary relation symbol. There exists a
  $\fo{\N,+,<,R}$-formula $\delta(x)$ such that, for all
  $\set{\N,+,<,R}$-structure $\s$, if $R^{\s}$ is not ultimately
  periodic, then $\delta(x)^{\s}$ is expanding.
\end{theorem}
Note that, formally, Theorem \ref{theo:exp} is an easy consequence of
\cite[Theorem 3.7]{michaux}. Indeed, \cite[Theorem 3.7]{michaux}
states that there are two sets, which are $\fo{\N,+,<,R}$-definable,
and one of them is expanding. And furthermore, the definition of those
two sets does not depend on the interpretation of $R$.
\section{The theorem}\label{mva}
In this section, we prove the main theorem of this paper. It is similar to
\cite[Theorem 5.1]{michaux}, which is now recalled.
\begin{theorem*}[{\cite[Theorem
    5.1]{michaux}}]\label{theo:michaux-orig}
  Let $d\in\N^{>0}$ and $R\subseteq \N^{d}$. Then $R$ is
  $\fo{\N,+,<}$-definable if and only if every subset of \N{} which is
  $\fo{\N,+,<,R}$-definable is ultimately periodic.
\end{theorem*}
This theorem can be equivalently stated as follows.
\begin{corollary*}
  Let $d\in\N^{>0}$. Let $\s$ be a $\set{\N,+,<,R}$-structure such
  that $R^{\s}\subseteq\N^{d}$ which is not
  $\fo{\N,+,<}$-definable. There exists a $\fo{\N,+,<,R}$-formula
  $\nu_{R^{\s}}(x)$ such that $\nu_{R^{\s}}(x)^{\s}$ is not
  $\fo{\N,+,<}$-definable, i.e., not ultimately-periodic.
\end{corollary*}
Resuming Examples \ref{ex:muc} and
\ref{ex:muc2}, two examples of such sets are given.
\begin{example}\label{ex:michaux}
  Let $\V=\set{\N,+,<,R}$ where $R$ is a binary relational symbol.
 Let $\s^{0}$ be the $\V$-structure such that
    \begin{equation*}
      R^{\s^{0}}:=\set{(x_{0}^{2},x_{1})\mid x_{0},x_{1}\in\N}.
    \end{equation*}
    In this case, it suffices to consider the section $x_{1}=0$. Then
    the formula $\nu(x)=R(x,0)$, defines the set
    $\set{n^{2}\in\N\mid n\in\N}$ which is not ultimately periodic.
  \end{example}
\begin{example}\label{ex:michaux2}
    Let $\s^{1}$ be the $\V$-structure such that
    \begin{equation*}
      R^{\s^{1}}:=\set{(x_{0},x_{1})\in\N^{2}\mid \forExDeux}.
    \end{equation*}
    The set $R^{\s^{1}}$ is pictured in Figure \ref{fig:ex:muc}.  Each
    section of the form $x_{0}=c$ is ultimately periodic with period 2
    and each section of the form $x_{1}=c$ is finite.

    Let $X$ be the set of pairs $(x_{0},x_{1})$ such that
    $\cube{(x_{0},x_{1})}1{R^{\s^{1}}}=\set{(0,1)}$.  The first of
    those elements are pictured as the lower-left corner of the
    squares of Figure \ref{fig:ex:muc}. Then it can be shown that
    $X=\set{((c+1)^{2},c)\mid c\in 2\N}$. Hence the set $N$ of norms
    of elements of $X$ is $\set{c^{2}+3c+1\mid c\in2\N}$. Note that
    the set $X$ is not ultimately periodic. The set $N$ is defined by:
   \begin{equation*}
    \arraycolsep=0.5pt
     \begin{array}{llll}
       \nu(x):=
       \exists x_{0},x_{1}.
       x_{0}+x_{1}\doteq x&\land
       \neg R(x_{0},x_{1})\land
       R(x_{0},x_{1}+1)\\&\land
       \neg R(x_{0}+1,x_{1})\land
       \neg R(x_{0}+1,x_{1}+1).
     \end{array}
 \end{equation*}
\end{example}

The main theorem of this paper is now stated.
\begin{theorem}\label{prop:michaux}
  Let $d\in\N^{>0}$. Let $R$ be a $d$-ary relation symbol.  There
  exists a $\fo{\N,+,<,R}$-formula
  $\nu_{d}(x)$\index{nudx@$\nu_{d}(x)$ - The formula which defines a
    set which is not ultimately periodic} such that, for every
  $\{R,+,<\}$-structure ${\s}$, if $R^{\s}$ is not
  $\fo{\N,+,<}$-definable, then $\nu_{d}(x)^\s$ is not
  ultimately-periodic, hence not $\fo{\N,+,<}$-definable.
\end{theorem}

In order to prove this Theorem, two lemmas are first proven. The
first lemma allows to reduce the problem of generating a set which is
not ultimately periodic to a simpler case.
\begin{lemma}\label{lem:ult-per}
  Let $R$ be a binary relation and $\V=\set{+,R}$.  There exists a
  $\fo{\N,<,R}$-formula $\epsilon(x)$ such that, for every
  $\V$-structure $\s$, if
  % \begin{equation}\text{\parbox{0.85\textwidth}{$R_{n}=\set{m\in{}N\mid R^{\s}(n,m)}$ is
  %     ultimately periodic with minimal period $p_{n}\in \N^{>0}$}}\end{equation}
  % \begin{equation}\label{eq:R_n}
  %   \text{
  %    % \parbox{\.85\textwidth}{
  %       {$R_{n}=\set{m\in{}N\mid R^{\s}(n,m)}$ is
  %    ultimately periodic with minimal period $p_{n}\in \N^{>0}$}
  %     %}
  %   }
  % \end{equation}
  \statement{eq:R_n}{for all $n\in\N$, $R_{n}:=\set{m\in{}N\mid R^{\s}(n,m)}$ is
    ultimately periodic with minimal period $p_{n}\in \N^{>0}$,}
  and if :
  \begin{equation}\label{eq:P->infty}
    \lim_{n\to+\infty}p_{n}=+\infty,
  \end{equation} then
  $\epsilon(x)^{\s}$\index{epsilon@$\epsilon(x)^{\s}$ - the formula
    which defines a set which is not ultimately periodic, from a sequence
    of ultimately periodic sets.}  defines a set $E(R^{\s})$ which is
  not ultimately periodic.
\end{lemma}
Two examples of sets $R$ satisfying the hypothesis of this lemma are now
given. 
\begin{example}
  Let $\pi_{n}$ be the $n$-th prime integer. Let
  \begin{equation*}
    R=\set{(n,m)\in \N^{2}\mid \pi_{n}\mbox{ divides }m}.
  \end{equation*}
  Then $R_{n}=\set{m\in \N\mid \pi_{n}\mbox{ divides } m}$ is the set
  of multiple of $\pi_{n}$, its minimal periodicity $p_{n}$ is
  $\pi_{n}$. Thus
  $\lim_{n\to+\infty}p_{n}=\lim_{n\to+\infty}\pi_{n}=\infty$. Let
  $q_{n}=\Pi_{i=0}^{n}\pi_{i}$, it is the least positive integer such
  that $q\in R_{i}$ for all $i<n$.  The distance between $q_{n}$ and
  $q_{n+1}$ is greater than $\pi_{n+1}$, thus is not bounded. Hence
  the set $S=\set{q_{n}\mid n\in\N}$ is not ultimately periodic.  Note
  that the property $y=q_{n}$ is defined by the formula:
    \begin{equation*}
      \rho(n;y):=\minTu{y}{0<y\land \forall i\le n. R(i,y)}.
  \end{equation*}
  Recall that the notation $\minTu{y}{\phi}$ is introduced in Notation
  \eqref{min-exists}.  Finally, the set $S$ is defined by:
  \begin{equation*}
    \exists n.x\doteq\rho(n).
  \end{equation*}
\end{example}
A second example is now given, which is a variation of the first
example.
\begin{example}\label{ex:prime-dec} Let
  $R=\set{(n,m)\in \N^{2}\mid \pi_{n}\mbox{ divides }m+n^2, m>n}$. It is
  represented in Figure \ref{fig:R-prime-dec}.
  \begin{figure}
    \centering
    \begin{tikzpicture}
      \pgfmathsetmacro\maxAx{70}
      \begin{axis}[
        width=\textwidth, height=200pt,
        xmin=0, ymin=-1,
        xmax=\maxAx, ymax=5,
        xlabel=$m$,
        ylabel=$n$]
        \foreach \i in {0,2,...,\maxAx} {\addplot[draw=none,mark=o]coordinates{(\i,0)};}
        \foreach \i in {2,5,...,\maxAx} {\addplot[draw=none,mark=o]coordinates{(\i,1)};}
        \foreach \i in {6,11,...,\maxAx} {\addplot[draw=none,mark=o]coordinates{(\i,2)};}
        \foreach \i in {5,12,...,\maxAx} {\addplot[draw=none,mark=o]coordinates{(\i,3)};}
        \foreach \i in {6,17,...,\maxAx} {\addplot[draw=none,mark=o]coordinates{(\i,4)};}
        \foreach \i in {14,27,...,\maxAx} {\addplot[draw=none,mark=o]coordinates{(\i,5)};}
        \addplot[dashed]coordinates{(0,0.2) (60,0.2)};
        \addplot[dashed]coordinates{(2,1.2) (62,1.2)};
        \addplot[dashed]coordinates{(6,2.2) (66,2.2)};
        % (2,0) (4,0) (6,0)(8,0)(10,0)(12,0)(14,0)(16,0)(18,0)(20,0)%2
          %                       %divise m+0
          % (2,1) (5,1) (8,1) (11,1) (14,1) (17,1) (20,1)%3 divise m+1
          % (1,2) (6,2) (11,2) (16,2) %5 divise m+4
          % (5,3)(12,3)(19,3) %7 divises m+9, i.e. m+2
          % (6,4)(17,4)%11 divises m+16, i.e. m+5
          % (1,5)(14,5) %13 divises m+25, i.e. m-1
%        };
        % \addplot[dashed] coordinates{(3,1)  (\xmax,1)};
        % \addplot[dotted] coordinates{(3,1)  (3,0)};
        % \addplot[dashed] coordinates{(4,4)  (\xmax,4)};
        % \addplot[dotted] coordinates{(4,4)  (4,0)};
        % \addplot[dashed] coordinates{(7,5)  (\xmax,5)};
        % \addplot[dotted] coordinates{(7,5)  (7,0)};
        % \node[draw=black,circle] at (3,0){};
        % \node[draw=black,circle] at (4,0){};
        % \node[draw=black,circle] at (7,0){};
      \end{axis}
    \end{tikzpicture}
    \caption{The set $R$ of Example \ref{ex:prime-dec}.}
    \label{fig:R-prime-dec}
  \end{figure}
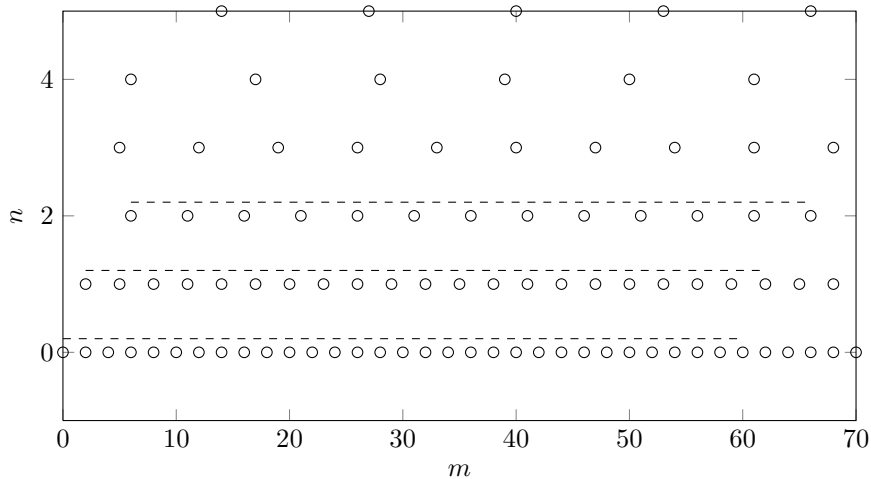
  Let $R_{n}=\set{m\in\N\mid R(n,m)}$. It is equal to
  $\left(m\Z-n^2\right)\cap\N$ and its minimal period $p_{n}$ is also
  $\pi_{n}$.  A formula $\alpha(n,p)$ is now introduced, which states
  that $R_{n}$ is ultimately $p$-periodic.  Let:
  \begin{equation*}
    \alpha(n,p):=\exists t.\forall N>t. \left[R(n,N)\iff R(n,N+p)\right],
  \end{equation*}
  where $t$ represents the threshold, as defined in Definition
  \ref{def:ult-per}.  Let $q_{n}=\Pi_{i=0}^{n}\pi_{i}$, it is equal to
  $\lcm\set{\pi_{i}\mid i\in[n]}$, hence to
  $\lcm\set{p_{i}\mid i\in[n]}$. Thus $q_{n}$ can be defined by:
  \begin{equation*}
    \rho(n;q_{n}):=\minTu{q_{n}}{\forall i\le n. \alpha(n,q_{n})},
  \end{equation*}
  For example, $\pi_{0}=2$, $\pi_{1}=3$ and $\pi_{2}=5$, hence
  $q_{2}=60$. The dashed lines of Figure \ref{fig:R-prime-dec}, of
  length 60, illustrates the fact that the first three lines are
  ultimately $60$-periodic.  Finally, the set
  $S=\set{q_{n}\mid n\in\N}$ can still be represented as:
  \begin{equation*}
    \exists n.x\doteq\rho(n).
  \end{equation*}
\end{example}
Lemma \ref{lem:ult-per} is now proven.
\begin{proof}[Proof of Lemma \ref{lem:ult-per}]
  The set $E(R^{\s})$ is the set of least common multiples of the
  $p_{i}$'s, for $i\in[0,n]$ for $n\in\N$.  It must be shown that this
  set is not ultimately periodic and that it is
  $\fo{\N,<,R}$-definable. Let us first show that it is not ultimately
  periodic.  For any integer $n$, let:
  \begin{equation}\label{def:qn}
    q_{n}:=\lcm\set{p_{i}\mid i\in[n]}.
  \end{equation}
  Note that $E(R^{\s})=\set{q_{n}\mid n\in\N}$.  It follows from
  Definition \eqref{def:qn} of $q_{n}$ that $p_{n}\le q_{n}$, hence
  %\begin{equation}\label{eq:q-lim}
   $ \lim_{n\to+\infty}p_{n}\le \lim_{n\to+\infty}q_{n}$.
  %\end{equation}
   Since % By Statement \eqref{eq:q-lim}
  % $\lim_{n\to+\infty}p_{n}\le\lim_{n\to+\infty}q_{n}$ and
  furthermore by Hypothesis \eqref{eq:P->infty}
  $\lim_{n\to+\infty}p_{n}=+\infty$, thus % :
  % \begin{equation*}\label{eq:q-inft}
    $\lim_{n\to+\infty}q_{n}=+\infty$,
  %\end{equation*}
    hence: \statement{eq:inf-q-ne}{There exists infinitely many
      integers $n$ such that $q_{n+1}\ne q_{n}$.}  It follows from
    Definition \eqref{def:qn} of $q_{n}$ that, for all $n\in\N$:
  \begin{equation*}\label{eq:grow_q_eq}
    q_{n+1}=\lcm\set{p_{i}\mid i\in[n+1]}= \lcm(\lcm\set{p_{i}\mid i\in[n]},p_{n+1})=\lcm(q_{n},p_{n+1}).
  \end{equation*}
  Since $q_{n+1}=\lcm(q_{n},p_{n+1})$, for all $n\in\N$, %:
  %\statement{eq:grow_q}{
  $q_{n+1}$ is either $q_{n}$ or is greater than $2q_{n}$. %}
  Furthermore, by statement \eqref{eq:inf-q-ne}, there exists
  infinitely many integers $n$ such that $q_{n+1}\ne q_{n}$. It
  follows that
  %\statement{eq:q_n-inf}{
  there exists infinitely many integers $n$ such that
  $2q_{n}\le q_{n+1}$. %}
  It implies that the set $E(R^{\s})=\set{q_{n}\mid n\in\N}$ is
  infinite and the distance between two successive elements is not
  bounded. Hence $E(R^{\s})$ is not ultimately periodic.

  \paragraph{}
  It remains to logically define $E(R^{\s})$.  A formula $\alpha(n,p)$
  which states that $p$ is a periodicity of $R_{n}$ is first defined.
  Let:
  \begin{equation*}
    \alpha(n,p):= \exists t.\forall N. t<N\implies{\left[R(n,N)\iff R(n,N+p)\right]}, 
  \end{equation*}
  where $t$ represents the threshold, as defined in Definition
  \ref{def:ult-per}.  It should be noted that, for an arbitrary finite
  set $F$, the value of $\lcm(F)$ does not seem to be
  $\fo{\N,+,<,F}$-definable.  In this case, $q_{n}$ is equivalently
  defined as the least integer $p$ such that for all $i\le n$, the set
  $R_{i}$ is ultimately $p$-periodic. That is, $q_{n}$ is defined by
  the $\fo{\N,+,<,R}$-formula:
  \begin{equation*}
    \rho(n;q_{n}):=\minTu{q_{n}}{\forall i\le n. \alpha(n,q_{n})}.
  \end{equation*}
  Recall that the notation $\minTu{x}{\phi}$ is introduced in Notation
  \eqref{min-exists}.  Finally, the formula $\epsilon(x)$ which
  defines $E(R^{\s})$ is:
  \begin{equation*}
    \epsilon(x):=\exists n.x\doteq\rho(n).
  \end{equation*}
\end{proof}

The second lemma allows to transform a sequence of $d$ functions,
diverging to infinity, into a sequence of $d$ \emph{increasing}
functions diverging to infinity, by restricting the domain of the $d$
function to an infinite set of integers.
\begin{lemma}\label{lem:croissant}
  Let $d\in \N$, let $f_{0},\dots,f_{d-1}$ be unary function symbols
  and let $\V=\set{<,f_{0},\dots,f_{d-1}}$. There exists a
  $\fo{\N,\V}$-formula $\tau_{d}(t)$\index{tau@$\tau_{d}(t)$ - the
    formula which defines a domain to $d$ function such that those
    functions are increasing} such that, for every $\V$-structure
  $\s$, if:
  \statement{eq:hyp-f-incr}{$\lim_{t\to+\infty}f_{i}^{\s}(t)=+\infty$
    for all $i\in[d-1]$} then the set $\tau_{d}(t)^{\s}$ defines an
  infinite set $T\subseteq\N$ such that $\tu f$ is increasing of $T$.
\end{lemma}
Two  functions $f_{0}$ and $f_{1}$  are now given as example. It is
then explained how to define a
$\fo{\N,+,<,f_{0},f_{1}}$-formula  which defines an infinite set over
which both $f_{0}$ and $f_{1}$ are infinite.
\begin{example}
  Let $f_{0}$ be the function which sends the integer $n$ to
  $m\times n$, where $0<m\le 10$ and $n\equiv m\mod 10$.
  The first integers $f_{0}(n)$ are:
  \begin{equation*}
    \arraycolsep=2.8pt
    \begin{array}{l|llllllllllllllllllllllllllllllllllllllllllllllllllll}
      n&0&1&2&3&4&5&6&7&8&9&10&11&12&13&14&15&16&17&18&19&20\\
      \hline
      f_{0}(n)& 0&1&4&9&16&25&36&49&64&91&100&11&24&39&56&75&96&119&144&171&200
    \end{array}
  \end{equation*}
  Let us construct an infinite set $T_{0}$ over which $f_{0}$ is
  increasing. Clearly, $T_{0}$ can be taken to be $10\N+i$ for any
  $i\in[9]$.  Those ten sets $10\N+i$ are $\fo{\N,+}$-definable. Note
  that the set $10\N+1$ has the property that, for all $n\in 10\N+1$ and
  $n'\in \N$ if $n<n'$, then $f(n)<f(n')$.

  \paragraph{}
  A second example is now given.  Let $\mathcal P$ be the set of prime
  numbers. Then let $f_{1}(n)=\sum_{i\in\mathcal P\cap[n]}i$.
  The first integers $f_{0}(n)$ are:
  \begin{equation*}
    \arraycolsep=2.8pt
    \begin{array}{l|llllllllllllllllllllllllllllllllllllllllllllllllllll}
      n&0&1&2&3&4&5&6&7&8&9&10&11&12&13&14&15&16&17&18&19&20\\
      \hline
      f_{1}(n)& 0&0&2&5&5&10&10&17&17&17&17&28&28&41&41&41&41&58&58&77&77
    \end{array}
  \end{equation*}
  When
  $f_{1}$ is the only considered function, the set $T$ can be taken to
  be $\P$.  Note that this set is not
  $\fo{\N,+,<}$-definable. However, it can be defined as the set
  $\set{x\mid\forall y. x<y. f(x)\ne f(y)}$.

  Note that $f_{0}$ is not increasing on $T_{1}$. Indeed, 7 and 11
  belong to $T_{1}$ while:
  \begin{equation*}
    f(7)=7\times7=49>11\times1=f(11).
  \end{equation*}
  In order to consider simultaneously the functions $f_{0}$ and
  $f_{1}$, it suffices to replace the definition of $T_{1}$ by
  restricting the element to belong to $T_{0}$. That is, let $T$ be
  $\set{x\in T_{0}\mid\forall y. x<y. f(x)\ne f(y)}$.
\end{example}          
Lemma \ref{lem:croissant} is now proven.
\begin{proof}[Proof of Lemma \ref{lem:croissant}]
  The proof is by induction on $d$. For $d=0$, $T=\N$ can be
  chosen. Note that $\N$ is $\fo{\N,\V}$-defined by the formula
  $\tau_{0}$ equal to $\exists x. x\doteq x$.

  It is now assumed that $0<d$ and that the property holds for
  $d-1$. That is, it is assumed that there exists a
  $\fo{\N,<,f_{0},\dots,f_{d-2}}$-formula $\tau_{d-1}(t)$ such that
  for every $\s$ which satisfies the hypothesis, $\tau_{d-1}(t)^{\s}$
  is a set $T'\subseteq\N$ such that: \statement{eq:T'-inf}{ $T'$ is
    infinite} and \statement{eq:T'-incr}{the functions
    $f_{0},\dots,f_{d-2}$ are increasing on $T'$.}  Let $T$ be the set
  of elements of $T'$ such that the restriction of $f_{d-2}$ on
  $T'\setminus [t-1]$ is minimal on $t$. Geometrically speaking, $T$
  is the set of elements $t$ such that the graph of $f$ does not not
  cross the lines on the right of $(t,f(t))$. It is illustrated in
  Figure \ref{fig:f-T} with $T'=\N$. The dashed horizontal lines are
  starting at $(t,f(t))$ for $t\in T$. The half-circles on the $n$
  axis represents the elements of $T$.
  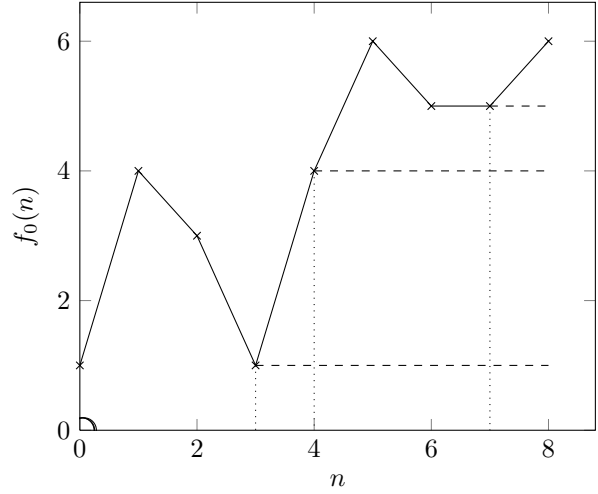
\begin{figure}
    \centering
    \begin{tikzpicture}
      \pgfmathsetmacro\xmax{8}
      \begin{axis}[
        xmin=0, ymin=0,
        xlabel=$n$,
        ylabel=$f_{0}(n)$]
        \addplot[mark=x] coordinates{
          (0, 1)
          (1,4)
          (2,3)
          (3,1)
          (4,4)
          (5,6)
          (6,5)
          (7,5)
          (8,6)
          % (9,)
          % (10,)
        };
        \addplot[dashed] coordinates{(3,1)  (\xmax,1)};
        \addplot[dotted] coordinates{(3,1)  (3,0)};
        \addplot[dashed] coordinates{(4,4)  (\xmax,4)};
        \addplot[dotted] coordinates{(4,4)  (4,0)};
        \addplot[dashed] coordinates{(7,5)  (\xmax,5)};
        \addplot[dotted] coordinates{(7,5)  (7,0)};
        \node[draw=black,circle] at (3,0){};
        \node[draw=black,circle] at (4,0){};
        \node[draw=black,circle] at (7,0){};
      \end{axis}
    \end{tikzpicture}
    \caption{A function $f$ and a set $T$ such that $T$ is increasing on
      $f$.}\label{fig:f-T}
  \end{figure}
  Formally, let:
  \begin{equation}\label{def:T}
    T:=\set{t\in T'\mid \forall t'\in T', t<t'\implies f_{d-1}(t)<f_{d-1}(t)}.
  \end{equation}
  Note that $T$ is defined by the $\fo{\N,\V}$-formula:
  \begin{equation*}
    \tau_{d}(t):=\tau_{d-1}(t)\land\forall
    t<t'. \left(\tau_{d-1}(t')\land t<t'\right) \implies f_{d-1}(t)<f_{d-1}(t).
  \end{equation*}
  Let us prove that $T$ satisfies the required condition, that is:
  $\tu f$ is increasing on $T$ and $T$ is infinite. Let us first prove
  that $\tu f$ is increasing on $T$. By Definition \eqref{def:T} of
  $T$, $T$ is a subset of $T'$ and by induction hypothesis
  \eqref{eq:T'-incr}, the functions $f_{0},\dots,f_{d-2}$ are
  increasing on $T'$. Thus: \statement{eq:T-incr-2}{the functions
    $f_{0},\dots,f_{d-2}$ are increasing on $T$.} It remains to prove
  that $f_{d-1}$ is increasing on $T$.  Let: \statement{def:ts}{$t<t'$
    be two elements of $T$.} In order to prove that $f_{d-1}$ is
  increasing on $T$, it remains to prove that $f_{d-1}(t)$ is smaller
  than $f_{d-1}(t')$. By Definition \eqref{def:ts} of $t'$, $t'\in T$,
  and by Definition \eqref{def:T} of $T$, the elements of $T$ belong
  to $T'$, thus:
  \begin{equation}\label{eq:t1T'}
    t'\in T'.
  \end{equation}
  By Definition \eqref{def:ts} of $t$, $t\in T$ and $t<t'$, by
  Equation \eqref{eq:t1T'}, $t'\in T'$ and by Definition \eqref{def:T}
  of $T$, for all $t'\in T'$ greater than $t$,
  $f_{d-1}(t)<f_{d-1}(t')$, thus 
%  \begin{equation}
    $f_{d-1}(t)<f_{d-1}(t')$.
%  \end{equation}
  Since $t<t'\in T$ implies that
  $f_{d-1}(t)<f_{d-1}(t')$, then:
  \statement{stat:f-1-incr}{$f_{d-1}$ is increasing.}  Having both
  Statements \eqref{eq:T-incr-2} and \eqref{stat:f-1-incr} implies
  that:\statement{stat:f-incr}{$\tu f$ is increasing on $T$.}
  \paragraph{}
  It remains to prove that $T$ is infinite. More precisely, it is
  proven that, for all $i\in\N$, $T$ contains at least $i$ elements.
  The proof is by induction on $i$. For $i=0$, it is trivial. Let
  $i\in\N^{>0}$ and let us assume that:\statement{T-i-1}{The set $T$
    contains a subset $T_{i-1}$ of cardinality $i-1$.}  In order to
  prove that $T$ contains at least $i$ elements, it suffices to define
  some integer $t$ and prove that $t\in T\setminus T_{i-1}$. The
  integer $t$ considered in the remaining of this proof is such that
  $\max{T_{i-1}}<t$, such that $f_{d-1}(t)$ is minimal under the
  preceding condition, and such that $t$ is maximal under the
  preceding conditions. Using the example of Figure \ref{fig:f-T} with
  $i=3$, $T_{2}=\set{3,4}$. The value of $t$ is then $7$.
  Geometrically, it is the one of lowest point $(t,f_{d-1}(t))$
  belonging to the graph on $f$ and on the right side of $(3,1)$ and
  $(4,4)$. Furthermore, between all of those minimal elements, it is
  the right-most one.
   
  It must be proven that this minimum and this maximum exists and the
  integer satisfying this definition belongs to $T\setminus T_{i-1}$.
  Let us first prove that $t$ is correctly defined.  In the remaining
  of the proof, let $\max\emptyset=-1$. This assumption allows to
  avoid to consider the case $i=1$ as special case. It is now proven
  that there exists a minimal integer $c$ of the form $f_{d-1}(t)$
  with $\max{T_{i-1}}<t$. % Let
  % \begin{equation}\label{def:m}
  %   m:=\max\set{f_{d-1}(t)\mid t\in T_{i-1}},
  % \end{equation}
  By Hypothesis \eqref{eq:hyp-f-incr},
  $\lim_{t\to+\infty}f_{d-1}(t)=+\infty$ and by Induction hypothesis
  \eqref{eq:T'-inf}, $T'$ is infinite, hence
  %\statement{eq:t>m-not-empty}'
  {\begin{equation*}\set{t \mid t\in T',\max{T_{i-1}}<t, \max\left(f_{d-1}(T_{i-1})\right)<f_{d-1}(t)}\end{equation*} is not empty,} and
  thus the image of $f_{d-1}$ on this set, % By Statement \eqref{eq:t>m-not-empty}
  % the set
  % $\set{t \mid t\in T',\max{T_{i-1}}<t, \max\left(f_{d-1}(T_{i-1})\right)<f_{d-1}(t)}$
  % is not empty,
  % thus
  %\statement{eq:f()-not-empty}
  %{$
  \begin{equation*} \set{f_{d-1}(t)\mid t\in
      T',\max{T_{i-1}}<t,\max\left(f_{d-1}(T_{i-1})\right)<f_{d-1}(t)},\end{equation*}%$
  is not empty. % }
  % By Statement \eqref{eq:f()-not-empty}, the set
  % $\set{f_{d-1}(t)\mid t\in T',\max{T_{i-1}}<t,\max\left(f_{d-1}(T_{i-1})\right)<f_{d-1}(t)}$
  % is not empty,
  Since this set is a non-empty subset of $\N$, it contains a minimal
  element $c$. Formally, let:
  \begin{equation}\label{def:c}
    c:=\min\set{f_{d-1}(t)\mid t\in T',\max{T_{i-1}}<t, \max\left(f_{d-1}(T_{i-1})\right)<f_{d-1}(t)}.
  \end{equation}
  As stated above, $c$ is the minimal integer of the form $f_{d-1}(t)$
  with $\max{T_{i-1}}<t$. It is now shown that there exists a maximal
  integer $t$, greater than $\max{T_{i-1}}$ such that $f_{d-1}(t)=c$.
  Note that it follows from Definition \eqref{def:c} of $c$ that:
  \begin{equation}\label{c>m}
    \max\left(f_{d-1}(T_{i-1})\right)<c.
  \end{equation}
  By Definition \eqref{def:c} of $c$, it follows that:
  \statement{eq:C-not-empty'}{
    $\set{t\mid t\in T',\max{T_{i-1}}<t, f_{d-1}(t)= c,
      \max\left(f_{d-1}(T_{i-1})\right)<f_{d-1}(t)}$
    is not empty.}  By Equation \eqref{c>m}
  $\max\left(f_{d-1}(T_{i-1})\right)<c$, thus having $f_{d-1}(t)= c$
  implies that $\max\left(f_{d-1}(T_{i-1})\right)<f_{d-1}(t)$. Hence
  Statement \eqref{eq:C-not-empty'} is equivalent to:
  \statement{eq:C-not-empty}{
    $\set{t\mid t\in T',\max{T_{i-1}}<t, f_{d-1}(t)= c}$ is not
    empty.}  By Hypothesis \eqref{eq:hyp-f-incr},
  $\lim_{t\to+\infty}f_{d-1}(t)=+\infty$, thus there exists an
  integer %\statement{eq:N}{
  $N\in\N$ such that for all $N<t$, $c<f_{d-1}(t)$.
%}  It follows from Statement \eqref{eq:N} that%:
  Hence %: \statement{eq:C-Finite}{
  $N$ is an upper-bound of the set
  $\set{t\mid t\in T',\max{T_{i-1}}<t, f_{d-1}(t)= c}$. %} The set
  % $\set{t\mid t\in T', \max{T_{i-1}}<t,f_{d-1}(t)= c}$ is a set of
  % integers with an upper-bound by Statement \eqref{eq:C-Finite} and
  Since, furthermore, by Statement \eqref{eq:C-not-empty}, this set is
  not empty, it admits a maximal element $t$. Formally, let:
  \begin{equation}\label{eq:t}
    t:=\max\set{t\mid t\in T',\max{T_{i-1}}<t, f_{d-1}(t)= c}.  
  \end{equation}
  Note that:
  \begin{equation}\label{eq:f(t)=c}
    c=f_{d-1}(t),
  \end{equation}
  and that:
   \begin{equation}
     \label{eq:t>t_i+1}
    \max{T_{i-1}}<t. 
  \end{equation}
  Note that the integer $t$ satisfies the properties stated in the
  beginning of the proof. It is greater than $\max{T_{i-1}}$,
  $f_{d-1}(t)$ is minimal under the preceding condition and $t$ is
  maximal under the preceding conditions.  It is now proven that
  $t\in T\setminus T_{i-1}$.  By \eqref{eq:t>t_i+1},
  $\max{T_{i-1}}<t$, hence:
  \begin{equation}\label{eq:tnotinTi-1}
    t\not\in T_{i-1}.
  \end{equation}
  It remains to prove that $t\in T$. By Definition \eqref{def:T} of
  $T$, it suffices to prove that for all $t'\in T'$, $t<t'$ implies
  $f_{d-1}(t)<f_{d-1}(t')$.  Let: \statement{eq:t>T_i}{$t'\in T'$ such
    that $t<t'$.} By Equation \eqref{eq:t>t_i+1} $\max{T_{i-1}}<t$ and
  by Statement \eqref{eq:t>T_i} $t<t'$. Hence:
  \begin{equation}
    \label{eq:t>max(T_{i-1})}
    \max{T_{i-1}}<t'.
  \end{equation}
  % By Statement \eqref{eq:t>T_i}, $t'\in T'$ and by Equation
  % \eqref{eq:t>max(T_{i-1})} $\max{T_{i-1}}<t'$, hence, by Definition \eqref{def:T}
  % of $T$:
  % \statement{eq:f(t)>T}{
  %   $f_{d-1}(n)<f_{d-1}(t')$ for all $n\in T_{i-1}$.}
  Since $T_{i-1}\subseteq T$:
  \begin{equation}\label{eq:maxT_iinT}
    \max{T_{i-1}}\in T.
  \end{equation}
  By Statement \eqref{eq:maxT_iinT}, $\max{T_{i-1}}\in T$, by Equation
  \eqref{eq:t>max(T_{i-1})}, $\max{T_{i-1}}<t'$, thus  by definition
  \eqref{def:T} of $T$:
  \begin{equation}\label{eq:f(t)>m'}
    f_{d-1}\left(\max(T_{i-1})\right)<f_{d-1}(t').
  \end{equation}
  By Statement \eqref{stat:f-incr}, $f_{d-1}$ is increasing on $T$ and
  by Definition \eqref{T-i-1} of $T_{i-1}$, $T_{i-1}\subseteq T$,
  hence $f_{d-1}$ is increasing on $T_{i-1}$. It follows that
  \begin{equation}\label{eq:max-f}
    f_{d-1}\left(\max(T_{i-1})\right)=\max\left(f_{d-1}(T_{i-1})\right).  
  \end{equation}
  Using this equality, $f_{d-1}\left(\max(T_{i-1})\right)$ can be
  replaced by $\max\left(f_{d-1}(T_{i-1})\right)$ in Equation
  \eqref{eq:f(t)>m'}. It follows that:
  \begin{equation}\label{eq:f(t)>m}
    \max\left(f_{d-1}(T_{i-1})\right)<f_{d-1}(t').
  \end{equation}
  % % By Statement \eqref{eq:f(t)>m}, $f_{d-1}(t)>f_{d-1}(t')$,
  By Statement \eqref{eq:t},  $t$ is the maximal element of $T'$,
  greater than $\max{T_{i-1}}$ and such that $f_{d-1}(t)= c$. Since, 
  by definition \eqref{eq:t>T_i} of $t'$, $t'\in T'$ and $t<t'$, and
  since, by
  Equation \eqref{eq:t>max(T_{i-1})},
  $\max{T_{i-1}}<t$, it follows that:
  \begin{equation}\label{eq-t-c}
    c\ne f_{d-1}(t').
  \end{equation}
  By Statement \eqref{def:c}, $c$ is the minimal integer of the form
  $f_{d-1}(t)$, for $t\in T'$, with $\max{T_{i-1}}<t$, and
  $\max\left(f_{d-1}(T_{i-1})\right)<f_{d-1}(t)$. Since, by Definition
  \eqref{eq:t>T_i} of $t'$, $t'\in T'$, since, by Equation
  \eqref{eq:t>max(T_{i-1})}, $\max{T_{i-1}}<t'$ and since, by Equation
  \eqref{eq:f(t)>m} $\max\left(f_{d-1}(T_{i-1})\right)<f_{d-1}(t')$,
  it follows that:
  \begin{equation}\label{t-c}
    c\le f_{d-1}(t').
  \end{equation}
  By Equation \eqref{eq-t-c} $c\ne f_{d-1}(t')$ and by Equation
  \eqref{t-c} $c\le f_{d-1}(t')$ then:
  \begin{equation}\label{f(t)>c}
    c<f_{d-1}(t').
  \end{equation}
  By Equation \eqref{f(t)>c} $c<f_{d-1}(t')$ and by Equation
  \eqref{eq:f(t)=c}, $c=f_{d-1}(t)$, thus:
%  \begin{equation}
  $f_{d-1}(t)<f_{d-1}(t')$. 
  % \end{equation}
  Since, for all $t'\in T'$ with $t<t'$, $f_{d-1}(t)<f_{d-1}(t')$,
  then:
  \begin{equation}\label{eq:tInT}
    t\in T.
  \end{equation}
  Let $T_{i}=T_{i-1}\cup\set t$. By Equation \eqref{eq:tnotinTi-1},
  $t\not\in T_{i-1}$, and by Definition \eqref{T-i-1} of $T_{i-1}$,
  $T_{i-1}$ contains $i-1$ elements. Thus $T_{i}$ contains $i$
  elements. By Definition \eqref{T-i-1} of $T_{i-1}$, the set
  $T_{i-1}$ is a subset of $T$ and by Equation \eqref{eq:tInT},
  $t\in T$, thus $T_{i}\subseteq T$. Hence $T$ admits a subset with
  $i$ elements. Hence the induction hypothesis holds.
\end{proof}
Theorem \ref{prop:michaux} is now proved.
\begin{proof}
  The formula $\nu_{d}(x)$ is defined by induction on $d$. If $d=1$,
  by Lemma \ref{relun}, $R^{\s}$ is not ultimately periodic, hence it
  suffices to set $\nu_{1}(x)=R(x)$.  Let us assume that $d$ is at
  least 2 and that the theorem holds for $d-1$.  Let:
  \statement{eq:def-F0}{$\tilde{\mathcal P}$ be the set of structures
    $\s$ such that $R^{\s}$ is not $\fo{\N,+,<}$-definable.}  In this
  proof, many functions are defined.  Their domain is
  $\tilde{\mathcal P}\times\N^{d'}$ for some $d'\in\N$ and their
  codomains are either the set of integers, the set of tuples of
  integers, or the set of subsets of $\N$.  It is then shown that one
  of the set of integers is not ultimately periodic. When those
  functions are defined, $\fo{\N,+,<,R}$-formulas which define them
  are also given.
  
  \paragraph{}
  Let $\s\in \tilde{\mathcal P}$. By Definition \eqref{eq:def-F0} of
  $\tilde{\mathcal P}$, $R^{\s}$ is not $\fo{\N,+,<}$-definable. Thus,
  by Corollary \ref{theo-muchnik} %: \statement{eq:hyp-S}{
  the structure $\s$ satisfies one of the two properties of Corollary
  \ref{theo-muchnik}. %}
  Two cases must be considered, depending on which property of
  Corollary \ref{theo-muchnik} is satisfied. Let us first assume that
  $\s$ satisfies Property \eqref{muc-ori-recur} of Corollary
  \ref{theo-muchnik}.  That is, there exists $i\in[d-1]$ and $j\in\N$
  such that $\secti{R^{\s}}{i}{j}$ is not $\fo{\N,+,<}$-definable. In
  this case, the induction on $d$ can be used on the lexicographically
  minimal pair $(i,j)$ to generate the set $E(R^{\s})$. In Example
  \ref{ex:michaux}, $\s^{0}$ satisfies Property \eqref{muc-ori-recur}
  of Corollary \ref{theo-muchnik}.  And in Example \ref{ex:michaux},
  $\s^{1}$ satisfies Property \eqref{muc-ori-local} of Corollary
  \ref{theo-muchnik}.

  Let us define the formula $\nu_{d}$.  Let us assume that there
  exists a $\fo{\N,+,<,R}$-formula $\nu_{d,1}(x)$ such that, for all
  $\s$ which satisfies Property \eqref{muc-ori-local} of Corollary
  \ref{theo-muchnik}, the set $\nu_{d,1}(x)^{\s}$ is not ultimately
  periodic.  Let $\mu_{d,i}(s)$ be the $\fo{\N,+,<,R}$-formula which
  states that $\secti{R^{\s}}{i}{s}$ is $\fo{\N,+,<}$-definable. It is
  defined by the formula $\mucF{d-1}$ of Theorem
  \ref{much-orig-formula} where $R(n_{0},\dots,n_{d-2})$ is replaced
  by $R(n_{0},\dots,n_{i-1},s,n_{i},\dots,n_{d-2})$. Let
  $\nu'_{d-1,i}(x,s)$ be the formula $\nu_{d-1}(x)$ where
  $R(n_{0},\dots,n_{d-2})$ is replaced by
  $R(n_{0},\dots,n_{i-1},s,n_{i},\dots,n_{d-2})$ for all terms
  $\tu n$. Then let
  \begin{equation*}
    \nu_{d}(x):=
    \ite{
      \bigvee_{i=0}^{d-1}\exists s.\minTu{i\in[d-1],s}{\neg\mu_{d,i}(s)}
    }{
      \nu'_{d-1,i}(x,s)
    }{
      \nu_{d,1}(x)}.
  \end{equation*}
  Recall that the notation $\minTu{x}{\phi}$ is introduced in Notation
  \eqref{min-exists} and that the notation $\ite{\phi}{\psi}{\xi}$ is
  introduced in Notation \eqref{ite}.

  \paragraph{}
  It remains to construct the $\fo{\N,+,<,R}$-formula
  $\nu_{d,1}(x)$. It is now assumed that the structure $\s$ satisfies
  Property \eqref{muc-ori-local} of Corollary \ref{theo-muchnik}.
  That is: \statement{eq:s-local}{For every $s\in \N$, there exists
    $k(R^{\s},s)\in\N$ such that for every $t\in\N$, there exists
    $\tu{c}(R^{\s},s,t)\in\N^{d}$ with $t\le\min(\tu{c}(R^{\s},s,t))$
    such that the \notShiftable{R^\s}{\tu
      c(R^{\s},s,t)}{k(R^{\s},s)}{s}.}  As in Section
  \ref{sec:some-res}, for $s$ and $t$ fixed, $k(R^{\s},s)$ and
  $\tu{c}(R^{\s},s,t)$ denote the lexicographically minimal tuple of
  integer which satisfies Statement \eqref{eq:s-local}.  Recall that,
  by Lemmas \ref{lem:c} and \ref{lem:k}, they are defined by the
  $\fo{\N,+,<,R}$-formulas $\gamma_{d}(s,t;\tu{c})$ and
  $\kappa_{d}(s;K)$ respectively.  Note in particular that, by
  Statement \eqref{eq:s-local}, $t\le\min(\tu{c}(R^{\s},s,t))$, thus
  for all $s\in\N$ %: \statement{min-inf}{
  $\lim_{t\to+\infty}\min(\tu{c}(R^{\s},s,t))=+\infty$. %} % By Definition \eqref{eq:s-local} of $\tu{c}$:
  % \statement{eq:min>t}{$\min(\tu{c}(R^{\s},s,t))\ge t$, for all
  %   $s\in\N$ and $t\in\N$.} 
  Hence, for all $i\in[d-1]$ and for all $s\in\N$:
  \statement{c_i-inf}{$\lim_{t\to+\infty}(c_{i}(R^{\s},s,t))=+\infty$.}
  By Statement \eqref{c_i-inf}, for each $s\in\N$, the $d$ functions
  $c_{0}(R^{\s},s,\cdot)$, \dots, $c_{d-1}(R^{\s},s,\cdot)$ satisfy
  the hypothesis of Lemma \ref{lem:croissant}. Let:
  \statement{eq:def:T}{ $T(R^{\s},s)\subseteq{}\N$ be the set defined
    by Lemma \ref{lem:croissant} applied to the $d$ functions
    $c_{0}(R^{\s},s,\cdot)$, \dots, $c_{d-1}(R^{\s},s,\cdot)$,} and
  let: \statement{eq:def:tau}{$\tau_{d}(s,t)$ be the
    $\fo{\N,+,<,R}$-formula of Lemma \ref{lem:croissant} where each
    equality $y\doteq f_{i}(x)$ is replaced by the formula
    $\exists
    z_{0},\dots,z_{d-2}.\gamma_{d,i}(s,t;z_{0},\dots,z_{i-1},y,z_{i},\dots,z_{d-1})$.}
  By Lemma \ref{lem:croissant}, this formula defines $T(R^{\s},s)$.
  The set $T(R^{\s},s)$ is a set of indexes such
  that:\statement{c_incr_on_T}{ for all $t,t'\in T(R^{\s},s)$, $t<t'$
    implies $\tu{c}(R^{\s},s,t)<\tu{c}(R^{\s},s,t')$.}  The set
  $E(R^{\s})$ is extracted from $T(R^{\s},s)$.  For the structure
  $\s^{1}$ of Example \ref{ex:michaux2}, and for all $s\in\N^{>0}$,
  $T(R^{\s^{1}},s)= 2\N$, and, as explained in Example \ref{ex:muc},
  $\tu{c}(R^{\s^{1}},s,t)$ is of the form $((c+1)^{2},c)$.
  \paragraph{}
  The cubes of size $s$ at position $\tu{c}(R^{\s},s,t)$, are now
  considered.  For $s,t\in\N$, let:
  \statement{eq:def:K}{$K(R^{\s},s,t)$ denote the cube
    $\cube{\tu{c}(R^{\s},s,t)}{k(R^{\s},s)}{R^{\s}}$.}  The cube
  $K(R^{\s},s,t)$ is such that all of its coordinates are at least $t$
  and furthermore the \notShiftable{R^\s}{\tu
    c(R^\s,s,t)}{k(R^\s,s)}{s}.  For each $s\in\N$, by Definition
  \eqref{eq:def:K},% of
  % $K(R^{\s},s,t)$
  the set $\{K(R^{\s},s,t)\mid t\in T(R^{\s},s)\}$ is a set of subsets
  of $[k(s)-1]^{d}$, hence is finite. It implies that
  % \statement{eq-pigeon}{
  for each $s\in\N$, there exists some cube
  $K(R^{\s},s,t)\subseteq[k(s)-1]^{d}$ that appears infinitely often
  in the sequence $(K(R^{\s},s,t))_{t\in T(R^{\s},s)}$. %}
  More precisely, it % Statement \eqref{eq-pigeon}
  implies that: \statement{def:f(R,s)}{For each $s\in \N$, there
    exists an integer $f(R^{\s},s)\in T(R^{\s},s)$ such that
    $K(R^{\s},s,f(R^{\s},s))$ appears infinitely often in the sequence
    $(K(R^{\s},s,t))_{t\in T(R^{\s},s)}$.}  Similarly to the choice of
  value of $k$ and $\tu{c}$, the value of $f(R^{\s},s)$ is chosen
  minimal.  A $\fo{\N,+,<,R}$-formula $\phi(s;F)$ which states that
  $F=f(R^{\s},s)$, as in Definition \eqref{def:f(R,s)}, is now given.
  By Lemma \ref{lem:eq-cube}, there exists a $\fo{\N,+,<,R}$-formula
  $\beta_{d}(\tu x,\tu y,k)$ which states that the cube
  $\cube{\tu{x}}{k}{R^{\s}}$ is equal to the cube
  $\cube{\tu{y}}{k}{R^{\s}}$.  Then, let:
  \begin{equation*}
    \phi(s;F):= \minTu{F}{
      \tau_{d}(s,F)\land
      \forall t\in\N.\exists t'. t<t'\land \tau_{d}(s,t')\land
      \beta_{d}(\tu x(s,F),\tu x(s,t'),\kappa_{d}(s))
    }.
  \end{equation*}
  \paragraph{}
  A name is now given to this cube which appears infinitely
  often. Let:
  \begin{equation}\label{def:I}
    I(R^\s,s):=K(R^{\s},s,f(R^{\s},s)).
  \end{equation}
  For the structure $\s^{1}$ of Example \ref{ex:michaux2},
  $f(R^{\s^{1}},s)=0$, for all $s\in\N^{>0}$, and
  \begin{equation*}
    I\left(R^{\s^{1}},s\right)=\cube{\tu{C}(R^{\s^{1}},s,0)}{0}{R^{\s^{1}}}=\set{(0,1)}.
  \end{equation*}
  \paragraph{}
  The set $E(R^{\s})$ is extracted from the set $X(R^{\s},s)$ of
  indices of cubes equal to $I(R^{\s},s)$.  Formally, let:
  \statement{eq:def:x}{$X(R^{\s},s):=\{t\in
    T(R^{\s},s)\mid{}K(R^{\s},s,t)=I(R^{\s},s)\}$.}
  For the structure $\s^{1}$ of Example \ref{ex:michaux}, for all
  $s\in\N^{>0}$, $X(R^{\s^{1}},s)=T(R^{\s^{1}},s)=2\N$.  A
  $\fo{\N,+,<,R}$-formula $\xi_{d}(s,t)$ is now given, which states
  that $t\in X(R^{\s},s)$. Let:
  \begin{equation*}
    \xi_{d}(s,t):=\tau_{d}(s,t)\land\beta_{d}(\psi_{d}(s,t),\psi_{d}(s,\phi(s)),\kappa_{d}(s)).
  \end{equation*}
  \paragraph{}
  By Definition \eqref{def:f(R,s)} of $f(R^{\s},s)$, for all $s\in\N$:
  \statement{eq:X-infinite}{The set $X(R^{\s},s)$ is infinite.}  It
  follows from Statements \eqref{eq:X-infinite} and \eqref{c_i-inf}
  that, for all $s\in\N$: \statement{eq:c-infinite}{The set
    $\{\tu{c}(R^{\s},s,t)\mid t\in X(R^{\s},s)\}$ is also infinite.}
  The set $E(R^{\s})$ is constructed from
  $\{\tu{c}(R^{\s},s,t)\mid t\in X(R^{\s},s)\}$. Note however that it
  is a set of tuples of integers and not a set of integers. In order
  to consider a set of integers, the set of norms is now considered.
  For all $s\in\N$, let
  \begin{equation}\label{def:N()}
    N(R^{\s},s):=\set{\norm{c(R^{\s},s,t)}\mid{t\in{X(R^{\s},s)}}}.
  \end{equation}
  For the structure $\s^{1}$ of Example \ref{ex:michaux},
  $N(s)=\set{c+(c+1)^{2}\mid c\equiv0\mod 2,\ \fracInline{s}4<c}$.  A
  $\fo{\N,+,<,R}$-formula $\zeta_{d}(s,x)$ is given, which states that
  $x\in N(R^{\s},s)$, as in Definition \eqref{def:N()}.  Recall that
  $\gamma_{d}(s,t)$ is the formula of Lemma \ref{lem:c} which defines
  $\tu{c}(R^{\s},s,t)$.  Then let:
  \begin{equation}\label{eq:zeta}
    \zeta_{d}(s,x):=\exists t. \xi_{d}(s,t) \land x\doteq\norm{\tu \gamma_{d}(s,t)},
  \end{equation}
  \paragraph{}
  Two cases must be considered, depending on whether there exists some
  $s$ such that $N(R^{\s},s)$ is not ultimately periodic or whether
  for all $s$, $N(R^{\s},s)$ is ultimately periodic.  If there exists
  an integer $s$ such that $N(R^{\s},s)$ is not ultimately periodic,
  then, let $E(R^{\s})$ be $N(R^{\s},s)$. As usual, the integer $s$ is
  assumed minimal.  % Let: \statement{eq:def:F2}{$F^{2}$ be the set of
    % structures $\s\in F^{1}$ such that for all $s\in\N$, the set
    % $N(R^{\s},s)$ is ultimately periodic.}

  The formula $\nu_{d,1}(x)$ is now defined.  Let us assume that there
  exists a $\fo{\N,+,<,R}$-formula $\nu_{d,2}(x)$ such that
  $\nu_{d,2}(x)^{\s}$ is not ultimately periodic, assuming that for
  all $s$, $N(R^{\s},s)$ is ultimately periodic.  If there is $s\in\N$
  such that some $N(R^{\s},s)$ is not ultimately periodic then
  $\nu_{d,1}(x)$ defines $N(R^{\s},s)$ with $s$ minimal.  Otherwise
  $\nu_{d,1}(x)$ uses the formula $\nu_{d,2}(x)$. Let $\mucF{1}'(s)$
  be the formula which states that $N(R^{\s},s)$, it is the formula
  $\mucF{1}$ of Theorem \ref{much-orig-formula}, where $R(x)$ is
  replaced by $\zeta_{d}(s,x)$.  Finally, let:
  \begin{equation*}
    \nu_{d,1}(x):=\ite{\exists s.\minTu{s}{\neg\mucF{1}'(s)}}{\zeta_{d}(s,x)}{\nu^{2}_{d}(x)}.
  \end{equation*}
  \paragraph{}
  It remains to construct the formula $\nu_{d,2}(x)$. It is now
  assumed that: \statement{eq:N-period}{ For all $s$, there exists an
    integer $p(R^{\s},s)$ such that $N(R^{\s},s)$ is ultimately
    $p(R^{\s},s)$-periodic.  } Similarly to the choice of value of
  $k$, $\tu{c}$ and $f$, the integer $p(R^{\s},s)$ is the minimal
  integer which satisfies Statement \eqref{eq:N-period}.  The set
  $E(R^{\s})$ is the set constructed by Lemma \eqref{lem:ult-per}
  applied to the set $\{(s,n)\mid n\in N(R^{\s},s)\}$.  The
  $\fo{\N,+,<,R}$-formula $\nu_{d,2}(x)$ is just the formula
  $\epsilon(x)$ of Lemma \ref{lem:ult-per}, where $R(s,x)$ is replaced
  by the formula $\zeta_{d}(s,x)$.

  It remains to prove that Lemma \eqref{lem:ult-per} can be applied to
  this set, that is, that $\lim_{s\to\infty }p(R^{\s},s)=\infty$. More
  precisely, it is proven that $s<p(R^{\s},s)$. In order to do this,
  it suffices to prove that $N(R^{\s},s)$ is infinite and that the
  distance between two distinct elements belonging to $N(R^{\s},s)$ is
  at least $s$.  Let: \statement{eq:def.s,t,'}{$s\in \N$ and
    $t,t'\in X(R^{\s},s)$ such that $t<t'$.}  It follows, by
  Definition \eqref{eq:def:x} of $X(R^{\s},s)$, that:
  \begin{equation}\label{eq:eq:K}
    K(R^{\s},s,t)=K(R^{\s},s,f(R^{\s},s))=K(R^{\s},s,t').
  \end{equation}
  By Definition \eqref{eq:def:K} of $K(R^{\s},s,t)$, it implies:
  \begin{equation}\label{eq:eq:c}
    \cube{\tu{c}(R^{\s},s,t)}{c(R^{\s},s)}{R^{\s}}=
    \cube{\tu{c}(R^{\s},s,t')}{c(R^{\s},s)}{R^{\s}}.
  \end{equation}
  By Definition \eqref{eq:s-local} of $\tu{c}(R^{\s},s,t)$, % :
  % \statement{eq:C-not-s}{
  the \notShiftable{R^\s}{\tu c(R^\s,s,t)}{k(R^s,s)}{s}. %}
  Hence, by  Equation \eqref{eq:eq:c}:%, and by Statement \eqref{eq:C-not-s},
  \statement{eq:shift-disjonction}{$\tu{c}(R^{\s},s,t)=
    \tu{c}(R^{\s},s,t')$
    or $s<\max(\abs{\tu{c}(R^{\s},s,t')-\tu{c}(R^{\s},s,t)})$.}  By
  Definition \eqref{eq:def.s,t,'} of $t$ and $t'$,
  $t,t'\in X(R^{\s},s)$, and by definition \eqref{eq:def:x} of
  $X(R^{\s},s)$:
  \begin{equation}\label{eq:tt'InT}
    t,t'\in T(R^{\s},s)
  \end{equation}
  By Statement \eqref{eq:tt'InT}, $t,t'\in T(R^{\s},s)$, by Definition
  \eqref{eq:def.s,t,'} of $t$ and $t'$, $t<t'$, by Statement
  \eqref{c_incr_on_T}, $t\mapsto\tu c(R^{\s},s,t)$ is increasing on
  $T(R^{\s},s)$, thus:
  \begin{equation}\label{eq:c<c}
    \tu{c}(R^{\s},s,t)<\tu{c}(R^{\s},s,t').
  \end{equation}
  It follows trivially that:
  \begin{equation}\label{eq:cnec}
    \tu{c}(R^{\s},s,t)\ne\tu{c}(R^{\s},s,t').
  \end{equation}
  By Equation \eqref{eq:cnec} and by Statement
  $\tu{c}(R^{\s},s,t)\ne\tu{c}(R^{\s},s,t')$ and by Statement
  \eqref{eq:shift-disjonction} (either
  $\tu c(R^{\s},s,t)=\tu c(R^{\s},s,t')$ or
  $c<\max(\abs{\tu c(R^{\s},s,t')-\tu{c}(R^{\s},s,t)})$), hence:
  \begin{equation}\label{eq:abs>s}
    s<\max(\abs{\tu c(R^{\s},s,t')-\tu{c}(R^{\s},s,t)}).
  \end{equation}
  By Equation \eqref{eq:c<c} $\tu{c}(R^{\s},s,t)<\tu{c}(R^{\s},s,t')$, thus:
  % and by Equation \eqref{eq:abs>s} $\max(\abs{\tu c(R^{\s},s,t')-\tu{c}(R^{\s},s,t)})>s$,
  % hence:
  \begin{equation}\label{eq:t'-t-abs}
    \tu{c}(R^{\s},s,t')-\tu{c}(R^{\s},s,t)=\abs{\tu{c}(R^{\s},s,t')-\tu{c}(R^{\s},s,t)}.
  \end{equation}
  By Equation \eqref{eq:t'-t-abs}
  $(\tu{c}(R^{\s},s,t')-\tu{c}(R^{\s},s,t))=\abs{\tu{c}(R^{\s},s,t')-\tu{c}(R^{\s},s,t)}$,
  replacing $\abs{\tu c(R^{\s},s,t')-\tu{c}(R^{\s},s,t)}$ by
  $\tu c(R^{\s},s,t')-\tu{c}(R^{\s},s,t)$ in Equation
  \eqref{eq:abs>s}, it follows that
  $% \begin{equation}\label{eq:t'-t>s}
  s<\max(\tu{c}(R^{\s},s,t')-\tu c(R^{\s},s,t)) $, % \end{equation}
  thus: \statement{def:i:c}{There exists $i\in[d-1]$ such that
    $s<c_{i} (R^{\s},s,t')- c_{i}(R^{\s},s,t)$.}  By Statement
  \eqref{eq:c<c} $0<c_{j}(R^{\s},s,t')- c_{j}(R^{\s},s,t)$ for all
  $j\in[d-1]$, and by Statement \eqref{def:i:c}
  $s<c_{i} (R^{\s},s,t')- c_{i}(R^{\s},s,t)$ for some $i\in[d-1]$. It
  follows that:
  \begin{equation}\label{eq:s<sum}
    s<\sum_{i=0}^{d-1}c_{i}(R^{\s},s,t')-c_{i}(R^{\s},s,t).
  \end{equation}
  Note that
  $\norm{\tu{c}(R^{\s},s,t')}-\norm{\tu{c}(R^{\s},s,t)}=\sum_{i=0}^{d-1}c_{i}(R^{\s},s,t')-c_{i}(R^{\s},s,t)$.
  Replacing $\sum_{i=0}^{d-1}c_{i}(R^{\s},s,t')-c_{i}(R^{\s},s,t)$ by
  $\norm{\tu{c}(R^{\s},s,t')}-\norm{\tu{c}(R^{\s},s,t)}$ is Equation
  \eqref{eq:s<sum} gives:
  \begin{equation}\label{eq:->s}
    s<\norm{\tu{c}(R^{\s},s,t')}-\norm{\tu{c}(R^{\s},s,t)}.
  \end{equation}
  \paragraph{}
  By Statement \eqref{eq:X-infinite}, $X(R^{\s},s)$ is infinite, and
  by Statement \eqref{eq:->s}, for $t$ distinct from $t'$ belonging to
  $X(R^{\s},s)$, $\norm{c(R^{\s},s,t)}\ne\norm{c(R^{\s},s,t')}$, thus:
  \statement{eq:N-infinite}{The set $N(R^{\s},s)$ is infinite.}  By
  Statement \eqref{eq:N-infinite} $N(R^{\s},s)$ is infinite and by
  Statement \eqref{eq:->s}, the difference between two integers of
  $N(R^{\s},s)$ is strictly greater than $s$. Hence:
  \statement{eq:N-not-periodic}{$N(R^{\s},s)$ is not ultimately
    $p$-periodic for any $p\le s$.}  By Definition \eqref{eq:N-period}
  of $p(R^{\s},s)$ and Statement \eqref{eq:N-not-periodic}, 
  % \statement{eq:p>s}{
  $s<p(R^{\s},s)$ for all $s\in\N$. %}
  Hence $
  % \begin{equation}
  %   \label{eq:lim:p}
  \lim_{s\to\infty }p(R^{\s},s)=\infty$. 
  % \end{equation}
  It follows that the set $\set{(s,i)\mid i\in N(R^{\s},s)}$ satisfies
  the hypothesis of Lemma \ref{lem:ult-per}.  It thus suffices to take
  $E(R^{\s})$ to be the set generated by this lemma and $\nu_{d,2}(x)$
  to be the formula $\epsilon(x)$ of Lemma \ref{lem:ult-per}, where
  $R(s,x)$ is replaced by the formula $\zeta_{d}(s,x)$.
\end{proof}
This theorem admits the following corollary.
\begin{corollary}\label{cor:michaux}
  Let $d\in\N^{>0}$. Let $R$ be a $d$-ary relation symbol.  There
  exists a $\fo{\N,+,<,R}$-formula $\nu_{d}'(x)$ such that, for every
  $\{R,+,<\}$-structure ${\s}$, if $R^{\s}$ is not
  $\fo{\N,+,<}$-definable then $\nu'_{d}(x)^{\s}$ is not expanding.
\end{corollary}
\begin{proof}
  It is straightforward from Proposition \ref{prop:michaux} and
  Theorem \ref{theo:exp}.
\end{proof}

\section{Conclusion}
In this paper, we proved that any set $R$ which is not
$\fo{\N,+,<}$-definable allows to $\fo{\N,+,<,R}$-define an expanding
set of integers, i.e. a set of integers which is not
$\fo{\N,+,<}$-definable.

We see two directions for further research. The first direction
consists in considering the same problem over other domains, such as
the reals, the rationals, or the finite domains.

The second direction consists in considering the same problem for
other vocabularies. In particular, the logic
$\fo{\N,\times,\text{ is a divisor of }}$ on the domain $\N$ is very
similar to $\fo{\N,+,<}$ on a domain of arbitrary dimension. Hence, it
may be possible to $\fo{\N,\times,\text{ is a divisor of }}$-define
some interesting set set of power of 2 using some sets which are not
$\fo{\N,\times,\text{ is a divisor of }}$-definable.

Let $\moda m$ be the set of predicates of the form $x\equiv i\moda m$.
By \cite{milchior-comput}, for each set $R$ which is not
$\fo{<,\moda m}$-definable, there exists a $\fo{<,R}$-formula which
defines a set which is not $\fo{<,\moda m}$-definable. It may be
interesting to apply methods introduced in this paper to this logic,
in order to obtain a formula independent from the interpretation of
$R$.

\bibliographystyle{alpha} \bibliography{../fo}
\renewenvironment{theindex}{\begin{multicols}{2}\begin{itemize}}{\end{itemize}\end{multicols}}
\printindex{}
\end{document}